\documentclass[4pt, oneside]{article}
\usepackage{amsmath,amsthm}
\usepackage[affil-it]{authblk}
\providecommand{\keywords}[1]{\noindent\textbf{\textit{Keywords:}} #1}

\usepackage{amssymb,stmaryrd,color,bm}
\newtheorem{theorem}{Theorem}[section]

\newtheorem{definition}[theorem]{Definition}
\newtheorem{lemma}[theorem]{Lemma}

\newtheorem{corollary}[theorem]{Corollary}
\newtheorem{proposition}[theorem]{Proposition}

\newtheorem{observation}[theorem]{Observation}

\newcommand\PsAlt\blacklozenge
\newcommand\NcAlt\blacksquare

\newcommand\yespr[1]{}
\newcommand\slang{\ensuremath{{\mathcal L}_2}}
\newcommand\flang{\ensuremath{{\mathcal L}_1}}

\newcommand\mlang{{\mathcal L}_C}

\newcommand\modelpar[1]{{\mathfrak M}_{#1}}
\newcommand\sat{{Sat}}
\newcommand\val{{\rm val}}
\newcommand\term{{\rm Trm}_C}
\newcommand\snt{{\rm Snt}_C}

\newcommand\nmodel{\ensuremath{\mathfrak N}}
\newcommand\model{\ensuremath{\mathfrak M}}

\newcommand\Oracleproof{{\tt Proof}^{\Lambda}_{T|X}}

\newcommand\oraclerule[6]{{\ensuremath{{\tt Rule}_{T|X}^{\Lambda}(#1,#2 ,#3,#4,#5, #6)}}\xspace}
\newcommand\provor[3]{[ #1\mathord| \mathord #2]^{\Lambda}_{#3}}
\newcommand\consor[3]{\la #1\mathord| #2\ra^{\Lambda}_{#3}}
\def\eca{\ensuremath{{{\rm ECA}_0}}\xspace}
\def\aca{{{\rm ACA}_0}}
\def\rcaa{{{\rm RCA}^\ast_0}}
\def\atr{{{\rm ATR}_0}}

\def\transin{{\tt TI}}
\def\compax{{\tt CA}}
\def\provfor{{\tt Proof}}
\def\Provfor{{\tt IPC}^{\Lambda}}

\newcommand{\prov}[1]{[{#1}]^{\Lambda}_T}
\newcommand{\provx}[2]{[{#1}]_{#2}}
\newcommand{\consx}[2]{\langle{#1}\rangle_{#2}}
\newcommand{\ORFN}{{\tt OracleRFN}}
\newcommand{\OCons}{{\tt OracleCons}}

\def\nc{{\Box}}

\def\<{\left\langle}
\def\>{\right >}

\newcommand{\glp}{{\ensuremath{\mathsf{GLP}}}\xspace}

\usepackage{xspace}

\newcommand{\pra}{\ensuremath{{\mathrm{PRA}}}\xspace}
\newcommand{\rca}{\ensuremath{{\mathrm{RCA}_0}}\xspace}

\newcommand{\pa}{\ensuremath{{\mathrm{PA}}}\xspace}
\newcommand{\Robinson}{\ensuremath{{\mathrm{Q}}}\xspace}

\newcommand{\AcaNaught}{\ensuremath{{{\rm ACA}_0}}\xspace}

\newcommand{\isig}[1]{{\ensuremath {\mathrm{ I}\Sigma_{#1}}}\xspace}

\newcommand{\ea}{\ensuremath{{\rm{EA}}}\xspace}

\newcommand{\la}{\langle}
\newcommand{\ra}{\rangle}

\newcommand\david[1]{}
\begin{document}

\title{Predicativity through transfinite reflection
      }

{
\author{Andr\'es Cord\'on--Franco}
\affil{Universidad de Sevilla\\
{\tt acordon@us.es}}
\author{David Fern\'andez--Duque}
\affil{Instituto Tecnol\'ogico Aut\'onomo de M\'exico\\
{\tt david.fernandez@itam.mx}
}
\author{Joost J. Joosten}
\affil{Universitat de Barcelona\\
{\tt jjoosten@ub.edu}}
\author{F\'elix Lara--Mart\'in}
\affil{Universidad de Sevilla\\
{\tt fflara@us.es}}

\maketitle%not jsl
}

\begin{abstract}
Let $T$ be a second-order arithmetical theory, $\Lambda$ a well-order, $\lambda<\Lambda$ and $X\subseteq \mathbb N$. We use $\provor \lambda XT\varphi$ as a formalization of ``$\varphi$ is provable from $T$ and an oracle for the set $X$, using $\omega$-rules of nesting depth at most $\lambda$''.

For a set of formulas $\Gamma$, define {\em predicative oracle reflection for $T$ over $\Gamma$} (${\tt Pred\mbox{-}O\mbox{-}RFN}_T[\Gamma]$) to be the schema that asserts that, if $X\subseteq\mathbb N$, $\Lambda$ is a well-order and $\varphi\in\Gamma$, then
\[
\forall\, \lambda{<}\Lambda\ (\provor \lambda XT\varphi\rightarrow \varphi).
\]
In particular, define {\em predicative oracle consistency} (${\tt Pred\mbox{-}O\mbox{-}Cons}(T)$) as ${\tt Pred\mbox{-}O\mbox{-}RFN}_T[\{{\tt 0=1}\}]$.

Our main result is as follows. Let $\atr$ be the second-order theory of Arithmetic Transfinite Recursion, $\rcaa$ be Weakened Recursive Comprehension and $\rm ACA$ be Arithmetic Comprehension with Full Induction. Then,
\[
\atr \equiv \rcaa +{\tt Pred\mbox{-}O\mbox{-}Cons}(\rcaa) \equiv \rcaa + {\tt Pred\mbox{-}O\mbox{-}RFN}_{{\rm ACA}} [{\bm \Pi}^1_2].
\]
We may even replace $\rcaa$ by the weaker $\eca$, the second-order analogue of Elementary Arithmetic.

Thus we characterize $\atr$, a theory often considered to embody Predicative Analysis, in terms of strong reflection and consistency principles.
\end{abstract}

{\keywords{Reflection principles; Fragments of second-order arithmetic; Predicative analysis; Reverse mathematics}}
%%%%%%%%%%%%%%%%%%%%%%%%%%%%%%%%%%%%%%%%%%%%%%%%%%%%%%%%%%

\section{Introduction}
Reflection over a theory $T$ roughly says that whatever is provable in $T$ is actually true. As such, reflection is natural from a methodological or philosophical point of view. Moreover, from a technical point of view it also turns out to be natural, since various well-known fragments of arithmetic can be re-cast in terms of reflection principles. In this introduction we will discuss results along these lines for first-order arithmetical theories and address the question of how this can be extended beyond first-order theories.

\subsection{Reflection, consistency and fragments of first-order arithmetic}
Fix a formal theory $T$. If we denote the formal provability of a formula $\varphi$ in $T$ by $\Box_T\varphi$, we can write ${\tt Rfn}(T)$, called \emph{local reflection over $T$,} as the scheme $\Box_T \varphi \to \varphi$, where $\varphi$ has no free variables. 

It turns out that a better-behaved notion of reflection is so-called \emph{uniform reflection}\david{I changed from `global reflection', I think this is more standard, isn't it?} where we allow for formulas, possibly with parameters. This scheme, denoted ${\tt RFN}(T)$, is given by 
\begin{equation}\label{EqRFN}
\forall x\ \big( \Box_T \varphi (\dot x) \to \varphi (x) \big),
\end{equation}
where $\varphi$ is any formula and $\dot x$ means that we must replace $x$ by a name for $x$.

Let us denote the consistency of $T+\varphi$ by $\Diamond_T \varphi$, which is equivalent to $\neg \Box_T \neg \varphi$. By G\"odel's Second Incompleteness Theorem we know that consistent computably enumerable (c.e.) theories $T$ do not even prove the weakest instances of reflection; if we define $\bot$ by $\tt 0=1$, we know that {\[T \nvdash \Box_T {\bot \to \bot},\]}
since the latter is provably equivalent to the consistency of $T$.

Thus, adding reflection to a consistent base theory will yield a proper extension of it. This is philosophically appealing, since one can sustain that it is natural to accept reflection over $T$ once one has accepted $T$. An early result by Kreisel and Levy \cite{KreiselLevy:1968:ReflectionPrinciplesAndTheirUse} shows that reflection principles are also natural from a technical point of view in that adding them yields natural extensions. Below, \pa denotes the well-known first-order theory \emph{Peano Arithmetic} and \pra refers to \emph{Primitive Recursive Arithmetic,} which is often considered to embody Hilbert's concept of finitist mathematics \cite{Tait:1981:Finitism}.

\begin{theorem}[Kreisel, Levy (1968)]
$\pra + {\tt RFN}(\pra) \equiv \pa$.
\end{theorem}

This relation between reflection and a system of arithmetic can be extended to fragments of Peano arithmetic. By $\Sigma_n^0$ --or simply $\Sigma_n$ in this introduction-- we denote first-order arithmetical formulas of the form $\exists x_n\forall x_{n-1}\hdots \varphi$, where $\varphi$ only has bounded quantifiers (see Section \ref{SubSecSynt}). The theory $\isig{n}$ is the theory of arithmetic where apart from the basic axioms for $+$ and $\times$ we have induction axioms \[\big( \forall \, y{<}x \ \varphi (y,z) \to \varphi(x,z)\big) \ \to \ \forall x\ \varphi(x,z)\]
for all $\Sigma_n$-formulas. 

Leivant \cite{Leivant:1983:OptimalityOfInduction} proved that there is a correspondence between the principles $\isig{n}$ for $n\geq 2$ and restricted reflection principles over \pra. Restricted reflection principles arise when we restrict the formulas in the reflection schema. If $\Gamma$ is a set of formulas, ${\tt RFN}_\Gamma(T)$, denotes the restriction of \eqref{EqRFN} where $\varphi (x) \in \Gamma.$ Beklemishev \cite{Beklemishev:1997:InductionRules} extended Leviant's result to the case $n = 1$ by lowering the base theory from \pra to a somewhat weaker theory  \ea called \emph{Kalmar Elementary Arithmetic}.

\begin{theorem}\label{theorem:LeivantsReflectionVersusInduction}
Given $n\geq 1$, $\ea + {\tt RFN}_{\Sigma_{n+1}}(\ea) \ \equiv \ \isig{n}$.\david{Yo agregu\'e el cuantificador sobre la $n$, por favor verificar que sea correcto.}
\end{theorem}

This theorem shows that reflection principles can be related to fragments of arithmetic. Moreover, for various notions of provability one can often link reflection to consistency statements. This relation is normally proved in the way presented in Lemma \ref{theorem:generalReflectionVersusConsistency} below, where we need to require some minimal properties of the particular provability predicate, leading to the notion of \emph{normal provability predicate}.

\begin{definition}
We will call a predicate\david{changed from `provability notion' since this is not a thing.} $\NcAlt$\david{This did not look right in the JSL format, but I made a macro so you can change it for something else.} a \emph{normal provability predicate} if it is provable in $\ea$ that $\NcAlt$ satisfies the modal logic $\sf K$; that is,\david{The definitions here were not coherent with their use later, so I changed them a bit, especially regarding the dependence on $U$.}
\[
\ea \vdash \NcAlt (\varphi \to \psi) \to (\NcAlt \varphi \to \NcAlt \psi),
\]
$\ea \vdash \NcAlt (\varphi)$ for any tautology $\varphi$, and $\ea \vdash \varphi$ implies that $\ea\vdash \NcAlt \varphi$.
\end{definition}

%The link between a particular notion of consistency and its corresponding notion of reflection is often rather straightforward and the proof normally follows the scheme of the following lemma.
For $\Gamma$ a set of formulas, let $\neg \Gamma$ denote the set $\{ \neg \gamma \mid \gamma \in \Gamma\}$. A predicate $\NcAlt$ is {\em provably $\Delta$-complete in $U$} if $U \vdash \delta \to \NcAlt \delta$ for any $\delta \in \Delta$.

\begin{lemma}\label{theorem:generalReflectionVersusConsistency}
Let $U$ be a theory extending elementary arithmetic and let $\NcAlt$ be a normal provability predicate with dual consistency predicate $\NcAlt$.
If $\Gamma$ contains (a provable equivalence of) $\bot$ and $\NcAlt$ is provably $\neg \Gamma$-complete in $U$, then $\Gamma$-reflection for $\NcAlt$ is equivalent to $\PsAlt \top$. That is,
\[
U + \{ \NcAlt \gamma \to \gamma \mid \gamma \in \Gamma \} \ \ \equiv \ \ U + \PsAlt \top.
\]
\end{lemma}

\begin{proof}
The $\vdash$ direction follows directly since $\bot \in \Gamma$. The other direction follows directly from $\neg \Gamma$-completeness: suppose that $\NcAlt \gamma$ and $\PsAlt \top$ and suppose for a contradiction that $\neg \gamma$; by completeness we have $\NcAlt \neg \gamma$, whence by $\NcAlt \gamma$ and normality we get $\NcAlt \bot$, which contradicts $\PsAlt \top$.
\end{proof}

Stronger notions of provability can be related to stronger notions of consistency. For this purpose it is very useful to consider the provability predicates $[n]_T$ for $n\in \mathbb N$ and c.e.~theories $T$ where $[n]_T$ is a natural first-order formalization of ``provable from the axioms of $T$ together with some true $\Pi_n$ sentence", where a formula is $\Pi_n$ if its negation is $\Sigma_n$. More precisely, let ${\tt True}_{\Pi_n}$ be the standard partial truth-predicate for $\Pi_n$ formulas, which is itself of complexity $\Pi_n$. Then, we define
\[[n]_T\varphi\leftrightarrow \exists \pi \ \big ( {\tt True}_{\Pi_n}(\pi)\wedge\nc_T (\pi\rightarrow \varphi) \big ).\]
It is well-known that each such predicate is normal and the following observation follows directly from the definition.

\begin{observation}\label{theorem:trivialPiNcompleteness}
For $T$ a c.e.~theory extending \ea and $n\in \mathbb N$, the predicate $[n]_T$ is $\Pi_n$-complete.
\end{observation}

From this observation, together with Lemma \ref{theorem:generalReflectionVersusConsistency}, we get that 
\begin{equation}\label{equation:nConsistencyVersusNreflection}
T + \la n\ra_T \top \ \equiv \ T + \{ \forall x\ \big( [n]_T \sigma(\dot x) \to \sigma(x) \big) \mid \sigma \in {\Sigma_n}\}.
\end{equation}
We can easily relate the reflection principle from this equation \eqref{equation:nConsistencyVersusNreflection} to a more standard notion of reflection in the following lemma.
\begin{lemma} For $T$ a c.e. theory extending \ea we have 
\[
T + \{ \, \forall x\ \big( [n]_T \sigma(\dot x) \to \sigma(x) \big) \mid \sigma \in {\Sigma_n}\} \ \equiv \ T + {\tt RFN}_{\Sigma_n}(T).
\]
\end{lemma}

\begin{proof}
``$\supseteq$" Reason in $T + \{ \, \forall x\ \big( [n]_T \sigma(\dot x) \to \sigma(x) \big) \mid \sigma \in {\Sigma_n}\}$ and fix some $\sigma\in \Sigma_n$ so that $\Box_T\sigma$. Then certainly $[n]_T \sigma$, whence $\sigma$, by reflection.

``$\subseteq$" Reason in $T + {\tt RFN}_{\Sigma_n}(T)$ and fix some $\sigma\in \Sigma_n$ such that $[n]_T\sigma$. Then, for some (possibly non-standard) $\Pi_n$ sentence we have that ${\tt True}_{\Pi_n}(\pi)$ and $\Box_T({\tt True}_{\Pi_n}(\pi) \to \sigma)$. Since ${\tt True}_{\Pi_n}(\pi) \to \sigma$ is provably equivalent to a $\Sigma_n$ sentence we obtain ${\tt True}_{\Pi_n}(\pi) \to \sigma$ by reflection. Since by assumption ${\tt True}_{\Pi_n}(\pi)$, we conclude $\sigma$.
\end{proof}

From this lemma, together with \eqref{equation:nConsistencyVersusNreflection} and Lemma \ref{theorem:LeivantsReflectionVersusInduction} we see the intimate relation between reflection, consistency statements and fragments of first-order arithmetic.

\begin{theorem}\label{theorem:intimateRelationConsistencyReflectionFragmentsArithmetic}
Given $n> 0$, \[\ea + \la n+1 \ra_\ea \top \ \equiv \  \ea + {\tt RFN}_{\Sigma_{n+1}}(\ea) \ \equiv \ \isig{n}.\]
\end{theorem}

\subsection{Ordinal analysis beyond first-order}
Ignatiev showed in \cite{Ignatiev:1993:StrongProvabilityPredicates} that for a large variety of theories $T$, the joint behavior of the provability predicates $[n]_T$ can be described and fully understood via a relatively simple and well-behaved modal logic called \glp, first studied by Japaridze in \cite{Japaridze:1988}. Theorem \ref{theorem:intimateRelationConsistencyReflectionFragmentsArithmetic} is a clear witness of the expressiveness of this modal logic; as a matter of fact, Beklemishev has shown in \cite{Beklemishev:2004:ProvabilityAlgebrasAndOrdinals} that the computation of a proof-theoretical $\Pi^0_1$ ordinal for \pa along the lines of Schmerl's work \cite{Schmerl:1978:FineStructure} on transfinitely adding consistency to a weak base theory can be realized largely within this modal logic \glp. 

In comparison to the more established $\Pi^1_1$ or $\Pi^0_2$ analyses of a formal theory $T$, its $\Pi^0_1$ analysis is more fine-grained and is more intimately tied to the notion of consistency. The ordinals involved in the analysis actually are naturally represented within the logic \glp by {\em worms,} which are iterated consistency statements of the form $\langle n_1\rangle\hdots\langle n_m\rangle\top$.

However, worms in \glp correspond to ordinals below $\varepsilon_0$, so \glp would certainly not suffice for the analysis of theories which are substantially stronger than \pa. Thus, to extend the relation between ordinals and modal-logical terms beyond $\varepsilon_0$, first steps were made in \cite{Beklemishev:2005:VeblenInGLP,BeklemishevFernandezJoosten:2014:LinearlyOrderedGLP, FernandezJoosten:2012:WellOrders} by studying logics $\glp_\Lambda$ that contain for each ordinal $\alpha < \Lambda$ a modality $[\alpha]$. Observation \ref{theorem:trivialPiNcompleteness} suggests that one should go beyond first-order theories in order to interpret $\glp_\Lambda$ and in \cite{FernandezJoosten:2013:OmegaRuleInterpretationGLP} two of the authors provide an interpretation of $\glp_\Lambda$ within second-order arithmetic.

Our present goal is to provide analogues of Theorem \ref{theorem:intimateRelationConsistencyReflectionFragmentsArithmetic} for fragments of second-order arithmetic, using more powerful reflection principles. These principles will require reasoning about arbitrary sets, which poses a challenge in a proof-theoretic setting. Every natural number can be represented by a closed term; however, it is of course not the case that each set can be denoted by a syntactical name in any countable language. 

In \cite{BeesonScedrov:1984:ChurchThesisEtAl} this problem was addressed by resorting to a richer language with sufficiently many names around. Our way to surpass this complication is by introducing and studying what we call \emph{oracle provability} which allows us to reason in a formalized setting about externally quantified sets. In particular, we will introduce a new first-order predicate $\mathfrak O$ which we shall call the {\em oracle predicate} and which will be used to reason formally about arbitrary sets.\david{Calling oracle provability a `trick' puts it in a negative light, so I have removed it.}

Once a workable relation between reflection principles and fragments of arithmetic has been established, this will constitute a significant step forward in the $\Pi^0_1$ ordinal analysis and semi-finitary consistency proofs of second-order theories as anticipated in \cite{Joosten:2013:AnalysisBeyondFO}. We believe that this paper makes an important step towards this goal.

\subsection{Overview of the paper}

In Section \ref{section:SecondOrderTheoriesWithOracles} we will settle our notation and nomenclature and fix the formal language that we will work in. One new ingredient is the first-order oracle symbol $\mathfrak O$ which will, in a way, provide a name for any arbitrary set. We define formalized oracle provability in this new language and mention some basic properties.

In Section \ref{section:FragmentsOfSOA} we will define the fragments of second-order arithmetic which are relevant for this paper and in section \ref{section:NestedOracleOmegaRules} we formalize within these fragments the notion $\provor\lambda XT$ of provability in $T$ using an oracle for $X$ and at most $\lambda < \Lambda$ iterations of the $\omega$-rule. Next, in Section \ref{section:OracleConsistencyAndOracleReflection} we define and prove the basic properties of the notions of reflection and consistency based the provability predicates $[\lambda|X]_T^\Lambda$.

In Section \ref{section:PredicativeReflectionAndConsistency} we define the notions of predicative oracle reflection and consistency. Basically, predicative oracle reflection/consistency is the statement that ``if $\Lambda$ is a well-order, then we have reflection/consistency for iterations of the $\omega$-rule along $\Lambda$ with an oracle for $X$''. The section proves that $\atr$, the second-order system of Arithmetic Transfinite Recursion, can be proven from predicative consistency. Finally, Section \ref{section:CountableCodedModels} proves the converse implication so that we end up with a characterization of $\atr$ in terms of predicative oracle reflection and conistency.

We know from the literature that $\pa \equiv \{ \la n \ra \top \mid n<\omega\}.$\david{$n\mbox{-}{\tt Con}(\ea)$ was used before, but never defined.} Our results show that $\atr$ can similarly be reduced to sufficiently strong consistency or reflection principles.
Using slightly suggestive notation, one can paraphrase our result as
\[
\atr \equiv \eca + ``\{  \alpha\mbox{-}{\OCons}(\eca) \mid \alpha \mbox{ a well-order} \}",
\]\david{No me encanta esta forma de escribirlo, pues da la impresi\'on de ser un esquema y no de tener una cuantificaci\'on interna, como es el caso.}
where $\eca$ is the second-order pendant of \ea.
\section{Second-order theories with oracles}\label{section:SecondOrderTheoriesWithOracles}
In this section we will define the fragments of second-order arithmetic which are relevant for this paper. Moreover, we will formalize a notion of provability where we use an oracle symbol which provides a name for an arbitrary set much like one has numerals to name natural numbers.
 
\subsection{Conventions of syntax}\label{SubSecSynt}

We will work mainly in the language $\slang$ of second-order arithmetic, with primitive symbols ${\tt 0},{\tt 1}, +_\omega ,\times_\omega ,  <_\omega , = , \in ,$ and a term $x^y$ for exponentiation representing the standard constants, operations and relations on the natural numbers. We use the subindex $\omega$ because we will be interested in different linear orders throughout the text and thus do not want to reserve these symbols for their standard use. However, we will omit such subindices whenever this does not lead to confusion.

We fix some primitive recursive G\"odel numbering, mapping a formula $\psi\in\slang$ to its corresponding G\"odel number $\ulcorner \psi \urcorner$, and similarly for terms and sequents of formulas (used to represent derivations). We will assume that the G\"odel numbering is natural in that the G\"odel number of a strict substring is numerically smaller than the G\"odel number of the entire string.

Moreover, we fix some set of \emph{numerals} which are terms so that each natural number $n$ is denoted by exactly one numeral written as $\overline n$. Since we will be working mainly inside theories of arithmetic, we will often identify $\psi$ with $\ulcorner \psi \urcorner$ or even with $\overline{\ulcorner \psi \urcorner}$ for that matter.

As is customary, we use ${\bm\Delta}^0_0$ to denote the set of all formulas (possibly with set parameters but without the occurrence of $\mathfrak O$) where no second-order quantifiers appear and all first-order quantifiers are ``bounded'', that is, of the form $\forall x<_\omega y \ \varphi$ or $\exists x<_\omega y \ \varphi$. We simultaneously define ${\bm\Sigma}^0_{0}={\bm\Pi}^0_{0}={\bm\Delta}^0_0$ and recursively define ${\bm\Sigma}^0_{n+1}$ to be the set of all formulas of the form $\exists x_0\hdots\exists x_m \varphi$ with $\varphi\in {\bm\Pi}^0_{n}$, and similarly ${\bm\Pi}^0_{n+1}$ to be the set of all formulas of the form $\forall x_0\hdots\forall x_m \varphi$ with $\varphi\in {\bm\Sigma}^0_{n}$. We denote by ${\bm\Pi}^0_\omega$ the union of all ${\bm\Pi}^0_n$; these are the {\em arithmetic formulas.} 

The classes ${\bm\Sigma}^1_n,{\bm\Pi}^1_n$ are defined analogously but using second-order quantifiers and setting ${\bm \Sigma}_0^1 = {\bm\Pi}^1_0 = {\bm\Delta}^1_0 = {\bm\Pi}^0_\omega$. It is well-known that every second-order formula is equivalent to another in one of the above forms.  We use a lightface font for the analogous classes where no set-variables appear free: $\Delta^m_n,\Pi^m_n,\Sigma^m_n$. The set of free variable of a formula $\varphi$ will be denoted by ${\sf FV}(\varphi)$.

We may also wish to enrich these classes with new formulas. If $\theta (x_1, \ldots, x_n)$ is any formula, we define $\Pi^m_n(\theta),\Sigma^m_n(\theta),$ etc. as above, except that we add any substitution instance $\theta (t_1, \ldots, t_n)$ as if it were an atomic formula. If $\theta$ is of the form $x\in Y$, we may write $\Gamma(Y)$ instead of $\Gamma(\theta)$. We may similarly define $\Gamma(\theta_1,\hdots,\theta_n)$ when adding multiple formulas as atomic. As a special case we mention $\Gamma(\mathfrak O)$, when the oracle is allowed to appear.

We will say a theory $T$ is {\em elementarily representable} or simply {\em representable} if it provably contains predicate logic, is closed under Modus Ponens, and there is a $\Delta^0_0$ formula ${\provfor}_T(x,y)$ which holds if and only if $x$ codes a derivation in $T$ of a formula coded by $y$.

We will also use the following pseudo-terms to simplify notation, where an expression $\varphi(t(\vec x))$ should be seen as a shorthand for $\exists y<_\omega s(\vec x)(\psi(\vec x,y)\wedge\varphi(y))$, with $\psi$ is a $\Delta^0_0$ formula\david{\textquestiondown $\Delta^0_0$ es correcto?} defining the graph of the intended interpretation of $t$ and $s$ a standard term bounding the values of $t(\vec x)$:

\begin{enumerate}

\item A term $\langle x,y\rangle$ which returns a code of the ordered pair formed by $x$ and $y$ and projection terms so that $(\langle x,y\rangle)_0 = x$ and $(\langle x,y\rangle)_1 = y$.

\item A term $x{[y/z]}$ which, when $x$ codes a formula $\varphi(v)$, $y$ a variable $v$ and $z$ a term $t$, returns the code of $\varphi(t)$. Otherwise, its value is unspecified, for example it could be the default $\ulcorner\bot\urcorner$.

\item A term $x\to y$ which, when $x,y$ are codes for $\varphi,\psi$, returns a code of $\varphi\to\psi$, and similarly for other Booleans and quantifiers.

\item A term $\overline x$ mapping a natural number to the code of its numeral.

\item For every formula $\varphi$ and variables $x_0, \ldots, x_m$, a term $\varphi(\dot x_0, \ldots, \dot x_m)$ which, given natural numbers $n_0, \ldots, n_m$, returns the code of the outcome of $\varphi[\vec x/\vec{\bar n}]$, i.e., the code of $\varphi(\bar n_0, \ldots, \bar n_m)$.

\item A designated unary predicate ${\mathfrak O}(\cdot)$ which we will use to represent an ``oracle''. 
\end{enumerate}
The only purpose of using these pseudo-terms is to shorten complex formulas for the sake of legibility. We write $\Box_T \varphi(\dot x_0, \ldots, \dot x_n)$ as shorthand for $\exists \psi\ (\psi = \varphi(\dot x_0, \ldots, \dot x_n) \wedge \Box_T \psi)$ and adhere to the same convention when using other notions of provability.

We will also use the notation for these terms in the metalanguage, as we have already seen with our notation for numerals. An exception to this is the usage of $\mathfrak O$ which will involved when representing sets in a formalized setting. All the classes of formulas $\Gamma$ we consider do not contain the symbol $\mathfrak O$ unless we explicitly mention otherwise, using the notation $\Gamma(\mathfrak O)$ defined below.

\subsection{Oracle provability}

The oracle $\mathfrak O$ will be used to add information about any set of numbers to our theory $T$.\david{Removed `indicator function' because we never really do represent it in function form.} To be precise, given a set $A\subseteq\mathbb N$, define $T|A$ to be the theory
\[
T \ + \ 
\{\mathfrak O(\overline n):n\in A\} \ + \
\{\neg\mathfrak O(\overline n):n\not\in A\} \ + \ \{ \exists X\, \forall x\, \big(x\,{\in}\, X \, \leftrightarrow \, \mathfrak O(x)\big)\}.
\]

It should be clear that if $T$ is representable then $T|A$ has a $\Delta^0_0(A)$ axiomatization\david{This used to say $\Delta^0_0(X,\mathfrak O)$. But the oracle appears {\em in} $T$, not in the description of $T$.} (i.e., the only set variable that may appear is $A$) and provability in $T|A$ is $\Sigma^0_1(A)$. To be more precise, there exist formulas ${\tt Axiom}_{T| A}(x,A) \in \Delta^0_0(A)$, as well as ${\provfor}_{T|A}(x,y,A) \in \Delta^0_0(A)$, and $\Box_{T| A}(x,A)\in \Sigma^0_1(X)$ so that 
\[
\begin{array}{llll}
\varphi \mbox{ is an axiom of } T| A & \Leftrightarrow & \mathbb{N} \models {\tt Axiom}_{T|A}(\ulcorner \varphi \urcorner,A), & \\
\pi \mbox{ is  a proof of } \varphi \mbox{ in the theory } T| A & \Leftrightarrow & \mathbb{N} \models {\provfor}_{T|A}(\ulcorner \pi \urcorner, \ulcorner \varphi \urcorner,A) & \mbox{ and, }\\
T | A \vdash \varphi & \Leftrightarrow & \mathbb{N} \models \Box_{T| A}(\ulcorner \varphi \urcorner,A).
\end{array}
\]
For example, we can define 
\[
\begin{array}{llllc}
{\tt Axiom}_{T| A}(\varphi,A) & := & {\tt Axiom}_{T}(\varphi)  & \vee & \exists\, x{<}\varphi\ (\varphi = \ulcorner \mathfrak O (\overline x) \urcorner \wedge x\in X) \\
 & & &  \vee  & \exists\, x{<}\varphi\ (\varphi = \ulcorner \neg \mathfrak O (\overline x) \urcorner \wedge x\notin X)\\
 & & &  \vee  &  \varphi = \ulcorner \exists X\, \forall x\, \big(x\,{\in}\, X \, \leftrightarrow \, \mathfrak O(x)\big)\urcorner,
\end{array}
\]
so that $\provfor_{T|A} (d,\varphi)$ will just become as $\provfor_{T} (d,\varphi)$ where every occurrence of ${\tt Axiom}_{T}(\pi)$ for some $\pi$ is replaced by ${\tt Axiom}_{T| A}(\pi,A)$.
% and where every occurrence of $y\in X$ that is not bounded by any second-order quantifier $\exists X$ or $\forall X$ is replaced by $\mathfrak O(y)$. 
As before, we define $\Box_{T|A} \varphi$ by $\exists x \ \provfor_{T|A} (x,\varphi)$. When working in $T|A$ we will write $x\in \overline A$ instead of $\mathfrak O(x)$, provided this does not lead to confusion. As is standard, one may use a single set to represent sets $A_1,\hdots,A_n$ by $\{\langle i,x\rangle:x\in A_i\}$, and as such we will freely use our oracle to interpret tuples of sets. If working on $T|A_1,\hdots,A_n$ we will write $x\in \overline A_i$ instead of $\mathfrak O(\langle i,x\rangle)$.

{
The oracle $\mathfrak O$ is a technical tool to be able to talk about sets under the scope of a box. For example, it is immediate that for any theory $T$, we provably have   
\[
\forall X \ \forall x \ \big (x \in X \ \rightarrow \ \Box_{T|X} \ \dot x \in \overline X\big) .
\]
However, a word of warning is due here. The ``bar'' notation we have introduced is informal and will only be used when it does not lead to ambiguity. In particular, nested uses of oracle provability could lead to confusion. Certain principles will go through for iterated oracle provability like $x\in X \to \Box_{T|X} \Box_{T|\overline X}\dot x \in {\overline X}$. Note that in this context, ``$\Box_{T|\overline X}$" refers to provability in the theory whose axioms are given by 
\[
\begin{array}{r}
{\tt Axiom}_{T}(\varphi) \vee \exists \, x{<}\varphi\ (\varphi = \ulcorner \mathfrak O (\overline x) \urcorner \wedge \mathfrak O (x)) \vee \exists \, x{<}\varphi\ (\varphi = \ulcorner \neg \mathfrak O (\overline x) \urcorner \wedge \neg \mathfrak O (x))\\
\vee \varphi = \ulcorner \exists X\, \forall x\, \big(x\,{\in}\, X \, \leftrightarrow \, \mathfrak O(x)\big)\urcorner.\\
\end{array}
\]
Since we use only one first-order symbol $\mathfrak O$ to write play the role of an oracle, in our framework the $X$ in $x\in X$ in for example $\Box_{T|X} \big( x\in X \to  \Box_{T|X}\dot x \in \overline X \big)$ is really given at the top level and not at ``one level of provability deep". As such we should be very careful in working with iterations of oracle provability and there are various other subtleties too. However, for the purpose of this paper these subtleties are not salient and we shall only be concerned with non-nested occurrences of provability.}

Finally, we mention that we would like to make inferences like $\Box_{T|X}\, \varphi (\overline X) \Rightarrow \Box_{T|X} \, \exists X\varphi (X)$. This is why in $T|X$ we have included the axiom $\exists X\, \forall x \ (x\in X \ \leftrightarrow \ \mathfrak O (x))$.

\section{Fragments of second-order arithmetic}\label{section:FragmentsOfSOA}

Important fragments of second-order arithmetic are typically characterized by their set-existence axioms. In this section we shall revisit the fragments which are relevant for this paper.

\subsection{Comprehension and Induction}

It is important in this paper to keep track of the second-order principles that are used; below we describe the most basic ones. Here, $<$ denotes the standard ordering on the naturals:
\begin{center}
\[\begin{array}{ll}
\Gamma\text{-}\compax&\exists X\forall x\ \big (x\in X\leftrightarrow \varphi(x)\big ) \ \ \ \ \ \ \ \ \ \hfill \text{ where }\varphi\in\Gamma \text{ and }X\text{ is not free in }\varphi;\\
{\tt I}\Gamma& \varphi({\tt 0})\wedge\forall x\, \big (\varphi(x) \to\varphi(x+ {\tt 1}) \big )\ \to\ \forall x \ \varphi(x) \hfill \text{ where }\varphi\in\Gamma;\\
{\tt Ind}& {\tt 0}\in X\wedge \forall x\ \big (x\in X\rightarrow x+{\tt 1}\in X \big )\ \to\ \forall x\, (x\in X).
\end{array}\]

\end{center}
We assume {\em all} theories extend two-sorted classical first-order logic, so that they include Modus Ponens, Generalization, etc., as well as Robinson's Arithmetic $\Robinson$, i.e. Peano Arithmetic without induction.

In the list below, recall that we have included exponentiation in our language, which is essential for the weaker theories.\david{I have just included exponentiation in all theories. It is needed for the weak ones and does not affect the strong ones.}

\begin{center}
\begin{tabular}{ll}
$\eca:$&\Robinson + ${\tt Ind}$+${\bm\Delta}^0_0$-$\compax$;\\
${\rm RCA}_0^\ast:$&\Robinson + ${\tt Ind}$+${\bm\Delta}^0_1$-$\compax$;\\
$\rca:$&$\Robinson + {\tt I}{\bm\Sigma}^0_1$+${\bm\Delta}^0_1$-$\compax$;\\
$\aca :$&$\Robinson + {\tt Ind}$+${\bm\Pi}^0_\omega$-$\compax$;\\
${\rm ACA} :$&$\Robinson + {\tt I}{\bm \Pi}^1_\omega$+${\bm\Pi}^0_\omega$-$\compax$.\\
\end{tabular}
\end{center}
We mention these from weakest to strongest but note that $\aca$ is still a relatively weak subsystem of second-order arithmetic, being arithmetcially conservative over \pa. It will later be useful to observe the following (see \cite[Lemma VIII.1.5]{Simpson:2009:SubsystemsOfSecondOrderArithmetic}):

\begin{theorem}\label{TheoFinAx}
$\rca$ and $\aca$ are finitely axiomatizable.
\end{theorem}

The system \eca stands for {\em Elementary Comprehension Axiom} and was introduced in \cite{FernandezJoosten:2013:OmegaRuleInterpretationGLP} as the second-order equivalent of Elementary Arithmetic \ea. In order to relate \eca to the more classical systems we need to mention $\Sigma_1^0$--bounding.
The principle of $\Sigma_1^0$--bounding is given by 
\[
\forall \, x{<}a\, \exists z \ \varphi(x,y,z) \ \to \exists b \, \forall \, x{<}a\, \exists \,z{<}b \ \varphi(x,y,z)
\]
with $\varphi \in \Sigma^0_1$, and is also referred to as $\Sigma^0_1$ collection. 

\begin{lemma}
The first-order part of \eca is \ea and the
first-order part of $\rcaa$ is \ea plus $\Sigma_1^0$--bounding.
Moreover, $\rcaa$ is $\Pi^0_2$-conservative over \eca but not $\Sigma^0_2$ conservative. 
\end{lemma}

\begin{proof}[Proof sketch]
In \cite{SimpsonSmith:1986:FactorizationOfPolynomials} the above-mentioned characterization of the first-order part of $\rca^\ast$ is proven. 

If $\model$ is a model of \ea then it is easy to check that $\la \model,S\ra$ is a model of \eca, where $S$ is the set of all $\Delta_0^0$-definable subsets of $\model$. 
This proves that the first-order part of \eca is just \ea.

It is well-known that $\ea$ plus $\Sigma^0_1$ collection is $\Pi^0_2$ conservative over \ea.\david{Reference?}
Moreover, the scheme of $\Delta_1^0$ minimal number principle is known to be $\Sigma^0_2$-axiomatizable, provable from $\Sigma^0_1$ collection over \ea, but not provable in \ea. This establishes that $\Sigma^0_2$ conservativity fails. \end{proof}

\subsection{Transfinite induction and transfinite recursion}\label{section:TransfiniteInductionAndTransfiniteRecursion}

Another principle that will be relevant in this work is {\em transfinite recursion,} but this is a bit more elaborate to describe. For simplicity let us assume that $\slang$ contains only monadic set-variables. Binary relations and functions are represented by coding pairs of numbers. We shall establish a few conventions for working with binary relations in second-order arithmetic.
%First, let us write {\em $R$ is a binary relation} and {\em $f$ is a function:}
%\begin{align*}
%{\tt rel}(R)&=\forall x\, \Big(x\in R\to\exists y\exists z\big(x=\langle y,z\rangle\big)\Big)\\
%{\tt funct}(f)&={\tt rel}(f)\wedge\forall x\exists! y\, \langle x,y\rangle\in f.
%\end{align*}
%Here, $\exists!$ is the standard abbreviation for {\em there exists a unique.}
%Also for simplicity, we may write $n\mathrel R m$ if $R$ represents a relation and $\langle n,m\rangle\in R$, or $n=f(m)$ if $\langle m,n\rangle\in f$ and $f$ is meant to be interpreted as a function. Further, it is possible to work with a second-order equality symbol, but it suffices to define $X\equiv Y$ by $\forall x\ (x\in X\leftrightarrow y\in Y)$.

We need to represent ordinals in second-order arithmetic. For this we will view a set $\Lambda$ as coding a pair $\langle |\Lambda|,<_\Lambda\rangle$. As is standard, we write $x<_\Lambda y$ instead of $\langle x,y\rangle\in \mathord <_\Lambda$. Let ${\tt linear}(\Lambda)$ be a formula naturally asserting that $\Lambda$ is a linearly ordered set, and define
\[\begin{aligned}
{\tt wo}(\Lambda)&={\tt linear}(\Lambda)\\
&\wedge\ \forall X\subseteq |\Lambda|\,  \Big(\exists x\in X \to \exists y\in X \, \forall z\in X \ \big(y\leq_\Lambda z\big)\Big).
\end{aligned}
\]
We will use lower-case Greek letters for elements of $|\Lambda|$, and boldface natural numbers to denote finite ordinals, so that ${\bm n}$ is $\{0,\hdots,n-1\}$ with the usual ordering. When it is clear from context that we are working within $\Lambda$ we may write $\xi<\zeta$ instead of $\xi<_\Lambda\zeta$, and similarly write $\xi<\Lambda$ instead of $\xi\in|\Lambda|$. A boldface $\bm \omega$ is the standard ordering on the naturals and a lightface $\omega$ denotes a natural number with order-type $\omega$ within a larger ordinal.

We will say an ordinal $\Lambda$ is {\em additive} if it comes equipped with a binary operation $+_\Lambda\colon |\Lambda|^2\to |\Lambda|$ satisfying the usual rules of ordinal addition; as before we may omit the subindex and just write $\alpha+\beta$ instead of $\alpha+_\Lambda \beta$. Note that additive ordinals also admit multiplication by natural numbers. Some of our results will be presented in additive ordinals, but this is not a problem over theories extending $\aca$ since one can readily replace $\Lambda$ by ${{\bm\omega}^\Lambda}$, which will remain provably well-ordered assuming that $\Lambda$ is.

We are interested in the theory $\atr$, in which new sets may be defined using {\em transfinite recursion.} Transfinite recursion is the principle that sets may be defined by iterating a formula along a well-order. To formalize this, let us consider a set $X$ whose elements are of the form $\langle \xi,x\rangle$. We shall write $x \in X_\xi$ for $\la \xi, x\ra \in X$.

Given a formula $\varphi(X)$, define $\varphi(X_{{<_\Lambda} \lambda})$ to be the formula where every occurrence of $t\in X$ in $\varphi$ is replaced by $(t)_0 <_\Lambda\lambda \ \wedge \  t\in X$. Then define ${\tt TR}^\Lambda(\varphi,X)$ to be the formula
\[
\forall \, \xi <\Lambda \, \forall x \ \Big(x\in X_\xi \ \leftrightarrow \  \varphi(x,X_{<_\Lambda\xi}) \Big).
\]
Finally, given a set of formulas $\Gamma$ we define the schema
\begin{center}
\begin{tabular}{lll}
${\tt TR}\text{-}\Gamma$:&$\forall \Lambda\ \Big( {\tt wo}({\Lambda})\rightarrow \exists Y\, {\tt TR}^{\Lambda}(\varphi,Y)\Big)$&for $\varphi\in\Gamma.$
\end{tabular}
\end{center}
We can now write down the axiom schema for Arithmetic Transfinite Recursion:
\begin{center}
\begin{tabular}{ll}
$\atr:$&\Robinson + $\tt Ind$+${\tt TR}$-${\bm\Pi}^0_\omega$.
\end{tabular}
\end{center}
The system ${\rm ATR}_0$ is commonly associated with Predicative Analysis and will be the main focus of this paper.\\
\medskip

In various proofs we will reason by induction along a well-order. By ${\transin}^\Lambda(\varphi)$ we denote the transfinite induction axiom for $\varphi$ along the ordering ${<_\Lambda}$:
\[
{\transin}^\Lambda(\varphi) \ := \Big ( \forall \, \xi <\Lambda \ \big(\forall \, \zeta <_\Lambda \xi \ \varphi(\zeta)\to \varphi(\xi)\big) \Big ) \to \forall \, \xi < \Lambda  \ \varphi (\xi).
\]
We will write $\varphi$-${\compax}$ instead of $\{\varphi\}$-${\compax}$, i.e., the instance of the comprehension axiom stating that $\{ x \mid \varphi(x) \}$ is a set. %Moreover, $\neg\varphi\text{-}{\compax}$ will stand for $\{\neg\varphi\}\text{-}{\compax}$ and not for $\neg(\varphi\text{-}{\compax})$.\david{No me gust\'o la convenci\'on de $\neg\varphi-CA$, adem\'as de que tiene poco sentido introducir una convenci\'on para un uso aislado.}
The following easy lemma tells us that we have access to transfinite induction for formulas of the right complexity (see \cite{FernandezJoosten:2013:OmegaRuleInterpretationGLP}):

\begin{lemma}\label{theorem:WellOrderingWithComprehensionImpliesTI}
In any second-order arithmetic theory containing predicate logic we can prove 
\[
{\tt wo}({\Lambda}) \wedge (\neg\varphi)\text{-}{\compax} \to {\transin}^\Lambda(\varphi).
\]
\end{lemma}

As an easy corollary to this lemma we see that in \AcaNaught we can apply transfinite induction for arithmetic formulas and we shall use this time and again in the remainder of this paper. 

{
We finish this section some remarks on variations of $\atr$.
To start, we can define Parameter-free Arithmetic Transfinite Recursion ${\rm ATR}^-_0$ as:
\begin{center}
\begin{tabular}{lll}
${\rm ATR}^-_0$:&\Robinson + $\tt Ind$+$\forall \Lambda\ \Big( {\tt wo}({\Lambda})\rightarrow \exists Y\, {\tt TR}^{\Lambda}(\varphi,Y)\Big)$& for $\varphi\in \Pi^0_\omega (\Lambda,Y)$.
\end{tabular}
\end{center}
So, the only difference between $\atr$ and ${\rm ATR}^-_0$ is that in $\atr$ we allow free set parameters in $\varphi$ other than $Y$ and $\Gamma$ while in ${\rm ATR}^-_0$ we do not. We thank Jeremy Avigad and Michael Rathjen for pointing out the following observation.

\begin{theorem}
${\rm ATR}^-_0 \vdash \atr$.
\end{theorem}

\begin{proof}
Reason in $\atr^-$. We fix some $\Lambda$ and assume ${\tt wo}(\Lambda)$. Now we wish to prove transfinite induction for some $\varphi(x,z,X,Y) \in \Pi^0_\omega(X,Y)$; as always, we may restrict ourselves to the case with just one additional set variable $Y$. Let $Y$ be arbitrary. Without loss of generality we may assume $Y\neq \varnothing$ so that we can fix some $n_0\in Y$. We now define $\Lambda' := \la |\Lambda'|, \prec' \ra $ with $|\Lambda'| := \{ \la x, y \ra \mid x\in |\Lambda| \wedge y\in Y \}$ and $\prec'$ on $|\Lambda'|^2$ as 
\[
\la x,y \ra \prec' \la x',y' \ra \ \ :\Leftrightarrow \ \ (x\prec x') \vee (x=x' \wedge y<y').
\]
Clearly, $\prec'$ defines a well-order if $\prec$ does. We now define $\varphi'$ to be obtained from $\varphi$ as follows: each occurrence of $y\in Y$ is replaced by $\exists x\ \la x,y\ra \in |\Lambda'|$; each occurrence of $x\prec y$ is replaced by $\la x,n_0\ra \prec \la y, n_0\ra$; and each occurrence of $x\in |\Lambda|$ is replaced by $\la x, n_0\ra \in |\Lambda'|$. It is easy to see that ${\tt TR}^{\Lambda'}(\varphi',X) \to {\tt TR}^\Lambda(\varphi,X;Y)$. Since $Y$ was arbitrary our result follows.
\end{proof}
From private correspondence with Michael Rathjen we learned that one can also consider a version of $\atr$ where one restricts the well-orders to ones which are definable via an arithmetical formula (allowing second-order parameters) without losing proof-strength. However, if one makes this restriction on ${\rm ATR}^-_0$ with moreover prohibiting set-parameters to occur in the arithmetical definition of the well-order, then a genuinely weaker system is obtained. One can even consider versions where no further first-order parameters are allowed as was done for bar induction in \cite{Rathjen:1991:RoleOfParametersInBarRuleAndInduction}. 
}

%%%%%%%%%%%%%%%%%%%%%%%%%%%%%%%%%%%%%%%%%%%%%%%%%%%%%%%%%%

\section{Nested $\omega$-rules for oracle provability}\label{section:NestedOracleOmegaRules}

In this section, we will briefly discuss how nested $\omega$-rules for oracle provability can be formalized and prove certain basic properties of the formalization.

\subsection{Formalizing nested $\omega$-rules for oracle provability} 

We will use $\prov{\lambda}\varphi$ to denote our representation of  ``$\varphi$ is provable in $T$ using one application of an $\omega$-rule of depth $\lambda$ (according to ${<_\Lambda}$)''.
The desired recursion for such a sequence of provability operators is given by the following equivalence:
\begin{equation}\label{equation:recursiveDefinitionProvability}
\prov\lambda\varphi \ \leftrightarrow 
\ \Big( \Box_T \varphi \vee \exists \, \psi\, \exists\, \xi{{<_\Lambda}} \lambda \ \big(\forall n \ \prov\xi\psi(\overline{n}) \ \wedge \ \Box_T (\forall x \psi (x) \to \varphi) \big) \Big).
\end{equation}

\noindent In \cite{FernandezJoosten:2013:OmegaRuleInterpretationGLP}, two of the authors formalize a notion $\prov\lambda\varphi$ in second-order number theory so that under certain conditions it provably satisfies the recursion from \eqref{equation:recursiveDefinitionProvability}.  

In the current paper we will need to slightly modify \eqref{equation:recursiveDefinitionProvability} to oracle provability. As such we will denote by $\prov{\lambda|X}\varphi$ the notion of iterated $\omega$-rule oracle-provability as defined by the following recursion.

\begin{equation}\label{equation:recursiveDefinitionOracleProvability}
\prov{\lambda|X}\varphi \ \leftrightarrow 
\ \Big( \Box_{T|X} \varphi \vee \exists \, \psi\, \exists\, \xi{{<_\Lambda}} \lambda \ \big(\forall n \ \prov{\xi|X}\psi(\overline{n}) \ \wedge \ \Box_{T|X} (\forall x \psi (x) \to \varphi) \big) \Big)
\end{equation}

The formalization of $\prov{\lambda|X}\varphi$ closely follows the presentation of \cite{FernandezJoosten:2013:OmegaRuleInterpretationGLP}. For the sake of clarity and keeping the paper self-contained we shall sketch here how such a formalization would proceed and refer for \cite{FernandezJoosten:2013:OmegaRuleInterpretationGLP} for further details. 

As a first step in such a formalization, we will use a set $P$ as a ``provability operator for the oracle $X$''. Its elements are codes of pairs $\langle\lambda,\varphi\rangle$, with $\lambda$ a code for an ordinal and $\varphi$ a code for a formula. We use $[\lambda]_P \varphi$ to denote $\langle \lambda,{ \varphi}\rangle\in P$.\david{I removed the notation $[\lambda|X]_P \varphi$ because it is not used consistently later.}

The idea is that we want to consider those sets $P$ of pairs $\langle\lambda,\varphi\rangle$ so that \eqref{equation:recursiveDefinitionOracleProvability} holds whenever we define $\prov{\lambda|X}\varphi := \langle\lambda,\varphi\rangle \in P$. Of course, we will use second-order logic to impose the necessary conditions on the set $P$. Whenever $P$ satisfies \eqref{equation:recursiveDefinitionOracleProvability} we will write $\Provfor_{T|X}(P)$ and say that ``$P$ is an iterated oracle provability class''. The following definition is a minor modification from \cite{FernandezJoosten:2013:OmegaRuleInterpretationGLP}.

%To be precise, $\lambda$ is in fact a natural number, where ${\Lambda}$ is a set-parameter whose intended interpretation is a well-ordering on the naturals. We then define a formula $\Provfor_T (X)$ as a formalization of\\
%
%\noindent {\em $\provx \lambda X \varphi$ if and only if either $\Box_T\varphi$ or there is a formula $\psi(x)$ and an ordinal $\xi{\leq_\Lambda}\lambda$ such that
%
%\begin{enumerate}
%\item for each $n<\omega$, $\provx\xi X{\psi(\overline n)}$, and
%
%\item $\Box_T(\forall x\psi(x)\to\varphi)$. 
%\end{enumerate}}
%\noindent We read $\Provfor_T(X)$ as ``$X$ is an iterated provability class''. Let us enter into a bit more detail:

\begin{definition}\label{definition:OracleRuleWithSecondOrderVariable}
Let $\Lambda$ denote a second-order variable that will be used to denote a well-order. Define $\oraclerule d \xi \lambda \psi \varphi P$ to be the formula
%\footnote{As a general convention, we will use a bar to separate parameters that are meant to be ``quantified away''; in Lemma \ref{LemmProvCode2} we will define a related formula ${\tt Rule}^{\leq_\Lambda}_T(d,\xi,\lambda,\psi,\varphi)$ which is independent of $X$.} 
\[ 
\phantom{aaaaa}\xi{<_\Lambda}\lambda\, \wedge \, \forall n\, \provx {\xi} P{\psi(\dot n)} \, \wedge \,  \provfor_{T|X}(d,\forall x \psi(x)\to \varphi),
\]
$\Oracleproof(c,\lambda,\varphi, P)$ to be 
\[
\exists d\, \exists\xi\, \exists\psi \ \Big(c=\langle d,\xi,\psi\rangle \, \wedge \, \big[ \provfor_{T|X} (d,\varphi)\, \vee \,  \oraclerule d \xi \lambda \psi \varphi P\big]\Big)
\]
and let $\Provfor_{T|X}(P)$ be the formula
\[
\forall \lambda\in|\Lambda| \ \forall \varphi \ \Big( \provx\lambda P\varphi \leftrightarrow \exists c\, \Oracleproof (c,\lambda,\varphi, P)\Big).
\]
Then, $\provor \lambda X T\varphi$ is the $\Pi^1_1$ formula $\forall P(\Provfor_{T|X}(P)\to \provx \lambda P{\varphi})$, and $\consor \lambda XT\varphi$ is defined as $\neg\provor \lambda X T{\neg\varphi}$.
\end{definition}

\noindent Note that the formulas $\provx  {\lambda} P\varphi$ and $\consx {\lambda} P\varphi$ are  independent of $T$, ${\Lambda}$ and $X$ and are merely of complexity ${ \Delta}^0_0(P)$.
Note also that for c.e.~theories $T$ we have that $\Provfor_{T|X}( P)$ is a ${\Pi}^0_3(\Lambda, X, P)$ formula.
%We can write the definition of $\Provfor_T(X)$ more succinctly as
%\[
%\Provfor_T(X) \ \leftrightarrow \ \forall \xi, \varphi \ \Big(\provx \xi X \varphi \ \leftrightarrow \ \big( \Box_T \varphi \vee \rulefor T \xi \varphi {\leq_\Lambda} X \big)\Big).
%\]

We would like to stress that our notion of oracle provability is rather weak in a sense: a formula is provable if it is a member of any iterated provability class corresponding to that oracle, but it can be the case that there simply are no such iterated provability classes. Consequently, and in the same sense, consistency statements are rather strong: from oracle consistency we may conclude the existence of an iterated provability class corresponding to that oracle. Let us state this explicitly in a lemma whose proof merely follows from the definition.

\begin{lemma}\label{theorem:consistencyImpliesIPC}
If $T$ is any representable theory, then
\[\eca\vdash \forall X\ \big(\consor \lambda X T\top \ \to \ \exists Y\ \Provfor_{T|X}(Y)\big).\]
\end{lemma}

However, if we are mainly interested in consistency strength, it is not such a bad thing to work under the assumption that iterated provability classes exists for all oracles.

\begin{lemma}\label{theorem:AddingExistenceIPCsRemainsEquiconsistent}
Let $T$ be a representable theory extending \eca, then $T$ and $T + \forall X\, \exists Y\ \Provfor_{T|X}(Y)$ are equiconsistent. That is,\[
T \equiv_{\Pi^0_1} T + \forall X\, \exists Y\ \Provfor_{T|X}(Y).
\]
\end{lemma}

\begin{proof}
As in \cite{FernandezJoosten:2013:OmegaRuleInterpretationGLP} we remark that $T + \Box_T\bot \ \vdash \ T + \forall X\, \exists Y\ \Provfor_{T|X}(Y) \ \vdash \ T$ and that $T + \Box_T \bot \equiv_{\Pi^0_1} T$. To see the latter, suppose $T \vdash \Box_T\bot \to \pi$ of some $\pi \in \Pi^0_1$, then by provable $\Sigma^0_1$-completeness also $T\vdash \Box_T \pi \to \pi$ whence by L\"ob's rule, $T\vdash \pi$.
\end{proof}

\subsection{Normality and completeness for oracle provability }

%We note that in our formalization of the recursion \eqref{equation:recursiveDefinitionOracleProvability} we iterated the omega-rule but there was no iteration of oracle-provability which we have seen was problematic.\david{Revise this.} Whenever we restrict to properties that do not involve iterations of formalized provability  we see that various results that hold for $\prov \lambda \varphi$ simply carry over to $\prov {\lambda |X} \varphi$. 

Many results established in \cite{FernandezJoosten:2013:OmegaRuleInterpretationGLP} for $\prov \lambda \varphi$ simply carry over to $\prov {\lambda |X} \varphi$. For example, via an easy transfinite recursion we see in \AcaNaught that iterated provability classes are unique, given a well-order $\Lambda$ and a c.e.~base theory $T$. 

\begin{lemma}\label{theorem:AcaNaughtProvesUniquenessIPCGivenWO}
For $T$ a c.e.~theory, we have that\[
\AcaNaught +{\tt wo}(\Lambda) \vdash \exists_{\leq 1}P\ \Provfor_{T|X}( P),
\]
where $\exists_{\leq 1}P\, \varphi(X)$ is an abbreviation of $\forall P \ \forall P' \ \big(\varphi(P)\wedge\varphi(P')\to P\equiv P'\big)$.
\end{lemma}

As an immediate consequence of the definition, we trivially get monotonicity:

\begin{lemma}\label{theorem:oracleProvabilityIsMonoton} It is provable in $\eca$ that if $\Lambda$ is a well-order and $\lambda' <_\Lambda \lambda$ then $\prov{\lambda'|X}\varphi \to \prov{\lambda|X}\varphi.$
\end{lemma}

We also see that we have distributivity for our provability predicates $\provor \lambda XT$.
\begin{lemma}[$\aca$]\label{kaxiom} Given a representable theory $T$, a well-order $\Lambda$, $\lambda < \Lambda$ and formulas $\varphi_1,\varphi_2$,
\[\provor \lambda XT(\varphi_1\to\varphi_2)\to(\provor \lambda XT\varphi_1 \to \provor \lambda XT\varphi_2).
\]
\end{lemma}

\proof
See \cite{FernandezJoosten:2013:OmegaRuleInterpretationGLP}. Note that the addition of an oracle does not affect the proof.
\endproof

\begin{corollary}
Within ${\aca} + {\tt wo}(\Lambda)$ the provability predicates $\provor \lambda XT$ are normal for each $\lambda<\Lambda$, provided $\ea\subseteq T$. 
\end{corollary}

We also can prove various completeness results for our provability predicates.

\begin{lemma}\label{LemmComplete}
If $\varphi\in\Sigma^0_{2m+1}(X,x)$ is arithmetic and $m\leq n$ then
\begin{equation}\label{EqLemmComplet}
\eca\vdash \forall X\, \forall x\ \Big(\varphi(X,x)\rightarrow [{\overline m}| X]^{\bm n}_T \, \varphi( \overline X, \dot x) \, \Big).
\end{equation}
\end{lemma}

\proof
Reasoning in $\eca$, we proceed by an external induction on $m$ and the subformulas of $\varphi$ by their build; to be precise, assume \eqref{EqLemmComplet} for each subformula of $\varphi$. Without loss of generality we may assume that negations occur only on atomic formulas, and that $\varphi$ does not contain subformulas of the form $\forall x\forall y\ \theta$ or $\exists x\exists y \ \theta$ where the occurrence of $y$ is unbounded.

For the base case, $\varphi$ is an atomic formula, which is of one of the following forms: either it contains no second-order variables, in which case we obtain $\nc_T\varphi(\dot x)$ by provable $\Sigma^0_1$-completeness. Otherwise, it is of the form $t\in X$ or $t\not\in X$ for some closed term $t$, which is provably equivalent to an axiom of $T|X$.

If $\varphi$ is bounded but not atomic, we merely follow a routine induction on the build of $\varphi$. The case where $\varphi$ is a Boolean combination of its subformulas is straightforward, and bounded quantifiers may be dealt with in a standard way, as for example in \cite{HajekPudlak:1993:Metamathematics}.

If $\varphi=\exists x \ \theta$, then for some $k$ we have that $\theta[x/\overline k]$ is true and we may use the induction hypothesis plus existential introduction. Finally, we consider the case $\varphi=\forall x \ \theta$. Since by assumption $\forall x \  \theta \in \Sigma^0_{2m+1}$ we have $\forall x \ \theta \in \Pi^0_{2m}$ since it starts with a universal quantifier so that $\theta \in \Sigma^0_{2(m-1)+1}$. Thus, by the induction hypothesis we have for every $k$ that $[\overline{m-1}| X]^{\bm n}_T\theta(\overline k)$ and therefore $[\bar m| X]^{\bm n}_T \forall x \ \theta (x).$
\endproof

As an immediate corollary we get $\Pi^0_{2m}$-completeness as well. We also get ${\bm\Sigma}^1_1$-completeness for provability involving at least $\omega$-many iterations as is manifest in the next theorem, provided we are allowed to pick a suitable oracle. We will write $[{ \omega}| X_1, \ldots, X_n]^{\Lambda}_T$ for oracle provability where the tuple $X_1, \ldots, X_n$ is represented by a single set using the conventions from Section \ref{section:FragmentsOfSOA}.

\begin{theorem}[$\eca$]\label{theorem:PiOneOneCompletenessForOracleProvability}
If $\varphi(X,x) \in{\Sigma}^1_{1}(X)$ and ${\bm \omega} \leq \Lambda$ then
\[\forall X \, \exists Y\   \Big( \varphi(X,x) \rightarrow [{ \omega}| Y, X]^{\Lambda}_T\, \varphi( \overline X, \dot x) \Big).\]
\end{theorem}

\begin{proof}
By assumption, $\varphi(X,x)$ is of the form $\exists Y \varphi_0(Y,X,x)$ with the formula $\varphi_0(Y,X,x) \in {\Pi}^0_\omega(X,Y)$. We now reason in $\eca$, fix some $X$ and $x$ and assume $\varphi(X,x)$. Thus, for some $Y$ we have $\varphi_0(Y,X,x)$. By Lemma \ref{LemmComplete} above together with monotonicity (Lemma \ref{theorem:oracleProvabilityIsMonoton}) we see that $[{ \omega}| Y, X]^{\Lambda}_T\, \varphi_0( \overline Y, \overline X, \dot x)$ whence also $[{ \omega}| Y, X]^{\Lambda}_T\, \varphi(\overline X, \dot x)$.%, quod erat demostrandum.\david{Creo que el cuadrito ya significa qed...}
\end{proof}

\section{Oracle consistency and oracle reflection}\label{section:OracleConsistencyAndOracleReflection}

In this section we shall define the notions of reflection and consistency that naturally correspond to oracle provability. Moreover, we shall link the two notions to each other and see how they relate to comprehension.

\subsection{Oracle consistency versus oracle reflection}

\begin{definition}[Oracle reflection and oracle consistency]
For $T$ a c.e.~ theory and $\Gamma$ a class of formulas not containing any occurrence of $\mathfrak O$, we define $\lambda$-$\ORFN^{\Lambda}_T[\Gamma]$ (oracle reflection) as the schema 
\[
\forall X_1 \ldots  \forall X_n \, \forall x\ \Big(\, \provor\lambda {X_1,\ldots, X_n}T \, \varphi(\overline {X_1}, \ldots, \overline{X_n}, \dot x) \to \varphi(X_1,\ldots, X_n,x) \Big)\]
for $\varphi(X_1,\ldots, X_n,x) \in\Gamma$ and ${\sf FV}(\varphi(X_1,\ldots, X_n,x)) \subseteq \{X_1,\ldots, X_n,x\}$. We define $\lambda$-$\OCons^{\Lambda}_T$ as $\lambda$-$\ORFN^{\Lambda}_T[\{ \bot \}]$.
\end{definition}

The next easy example makes clear why we imposed the restriction of $\mathfrak O$ not occurring in $\Gamma$. Let $X := \{ 2 \}$ and let $Y:=\varnothing$. Clearly we have $\provor \lambda X T \mathfrak O (\overline 2)$ and $\provor \lambda Y T \neg \mathfrak O (\overline 2)$ so that we arise at a contradiction were we allowed to apply oracle reflection to these sentences. We shall later see that oracle reflection with the restriction is consistent. Let us first see how consistency and reflection are easily related to each other.

\begin{lemma}\label{theorem:oracleConsistencyVersusOracleReflection}
Let $m>n$ denote natural numbers and let $\Lambda>\lambda \geq \omega$ denote ordinals. Over $\eca$ we have
\begin{enumerate}
\item 
$ n\text{-}{\ORFN}^{\bm m}_\eca [{\bm \Pi}^0_{2n+1}]\equiv n\text{-}{\OCons}^{\bm m}_\eca;$

\item
$\lambda\text{-}{\ORFN}^{\Lambda}_{\eca}[{\bm \Pi} ^1_1] \equiv \lambda\text{-}{\OCons}^{\Lambda}_{\eca}$.
\end{enumerate}
\end{lemma}

\begin{proof}
For the first item, it should be clear that $n\text{-}{\ORFN}^{\bm m}_\eca [{\bm \Pi}^0_{2n+1}]\vdash n\text{-}{\OCons}^{\bm m}_\eca$; but by ${\bm \Sigma}^0_{2n+1}$ completeness of $[{ n}| X]^{\bm m}_T$ provability as expressed in Lemma \ref{LemmComplete}, the latter also implies the former and all three are equivalent. For the third item this follows from ${\bm \Sigma}^1_1$-completeness (Theorem \ref{theorem:PiOneOneCompletenessForOracleProvability}). One has to check that the second-order quantifier from oracle provability does not essentially change the proof compared to Lemma \ref{theorem:generalReflectionVersusConsistency}.
\end{proof}

%
%\begin{lemma}\label{LemmConsRef}
%If ${\Lambda}$ has order-type at least $\omega$ then $\eca+{\ORFN}^{\Lambda}_{\eca}[\Pi^1_1]=\eca+{\OCons}^{\Lambda}_{\eca}.$
%\end{lemma}
%
%\proof
%The right-to-left inclusion is obvious. For the other, note that if $\forall X\varphi(X)$ is a false $\Pi^1_1$ formula then $\neg\varphi(X_0)$ holds for some $X_0$ and thus $\provor {n}{X_0}\eca\neg\varphi(\overline X_0)$ for some $n$, which means that we cannot have $\provor {\lambda}Y\eca \forall X\varphi(X)$ for otherwise $\provor {\max(n,\lambda)}{Y,X_0}\eca\bot$, contradicting consistency.
%\endproof

\subsection{Oracle reflection and comprehension}

We shall now see that with just a little amount of reflection we get arithmetical comprehension. 

\begin{lemma}\label{LemmRefComp}
$\aca\subseteq\eca+0\text{-}{\ORFN}^{\bm 1}_{\eca}[{\bm \Sigma}^0_1]$.\david{This is all we need for now. In the $\aca$ paper, we will put the other items.}
\end{lemma}

\proof
{Let $\vec X$ be a sequence of second-order variables $X_1,\ldots, X_n$ all different from $Y$. We only need to prove ${\bm\Sigma}^0_1$-$\tt CA$, which really amounts to $\forall \vec X\, \exists Y \, \forall y\ (y \in Y \, \leftrightarrow \, \varphi(x,\vec X))$ where $\varphi(x,\vec X)$ can be any formula in $\Sigma^0_1(\vec X)$. Let $\tilde \varphi(x,X)$ arise from $\varphi (x,\vec X)$ by replacing each occurrence of $x\in X_i$ by $\la i,x\ra \in X$. Clearly $\tilde \varphi$ is still a ${\bm \Sigma}^0_1$ formula and $\varphi$-$\tt CA$ is provably equivalent to $\tilde \varphi$-$\tt CA$ so that we may restrict our proof to formulas with a single free set parameter.}

Fix some set $X$. Since reflection implies consistency we know that there is some iterated provability predicate for $X$ (Lemma \ref{theorem:consistencyImpliesIPC}). That is, there is some $P$ with ${\tt IPC}_{\eca|X}^{\bm 1}(P)$ and moreover, since $\bm 1$ is provably well-founded we know that $P$ is unique (Lemma \ref{theorem:AcaNaughtProvesUniquenessIPCGivenWO}).
Let $\varphi(x,X)$ be a ${\Sigma}^0_1(X)$ formula. By $\Delta^0_0$-$\tt CA$ we can form the set
\[
Z=\{z \mid [ 0 | X]_P \, \varphi(\overline z, \overline X)\}.
\]
We claim that $z\in Z \leftrightarrow \varphi(z,X)$ which finishes the proof. If $z\in Z$, then by the uniqueness of $P$ we get $[0|X]^{\bm 1}_\eca \varphi(\overline z, \overline X)$, whence by reflection $\varphi(z, X)$. On the other hand, if $\varphi(z, X)$ we get by completeness (Lemma \ref{LemmComplete}) that $[0|X]^{\bm 1}_\eca \varphi(\overline z, \overline X)$ so that $z\in Z$.
\endproof

The upshot of this lemma is that we can perform many of our arguments in $\aca$ once we have just a little bit of reflection.

\section{Predicative reflection and consistency}\label{section:PredicativeReflectionAndConsistency}

Our notions of oracle reflection and oracle consistency depended on a particular well-order $\Lambda$. We shall now define what we call \emph{predicative oracle reflection} and \emph{predicative oracle consistency} which stipulate the oracle notions for any well-order.

\begin{definition}
We define the principle \emph{predicative oracle reflection}
\[
{\tt Pred\mbox{-}O\mbox{-}RFN}_T[\Gamma]=\forall \Lambda \, \forall\, \lambda\in|\Lambda|\ \big({\tt wo}({\Lambda})\to \lambda\text{-}{\ORFN}^{\Lambda}_T[\Gamma]\big),
\]
and \emph{predicative oracle consistency}
\[
{\tt Pred\mbox{-}O\mbox{-}Cons}(T)= \forall \Lambda\, \forall \, \lambda\in|\Lambda|\ \big ({\tt wo}({\Lambda})\to \lambda\mbox{-}\OCons^{\Lambda}_T \big).
\]
\end{definition}
Let us first make a simple observation.
\begin{lemma}\label{theorem:ECAoracleConsistencyVersusOracleReflection}
$\eca \vdash {\tt Pred\mbox{-}O\mbox{-}Cons}(\eca) \leftrightarrow {\tt Pred\mbox{-}O\mbox{-}RFN}_\eca[{\bm \Pi}^1_1]$.
\end{lemma}

\begin{proof}
This follows directly from Lemma \ref{theorem:oracleConsistencyVersusOracleReflection}.
\end{proof}

The main result of this section is that predicative analysis follows from predicative oracle consistency over a rather weak theory. To be more precise, we shall prove that $\atr$ follows from $\eca + {\tt Pred\mbox{-}O\mbox{-}Cons}(\eca)$. {In order to prove this we first shall first analyze $\atr$ in more detail.}

{\subsection{Predicative Analysis revisited}}

Let us recall from Section \ref{section:TransfiniteInductionAndTransfiniteRecursion} the formula ${\tt TR}^\Lambda(\varphi,X)$ which says that the set $X$ satisfies transfinite recursion for $\varphi$ over the well-order $\Lambda$. If $\varphi$ is arithmetical, then so is ${\tt TR}^\Lambda(\varphi,X)$. Moreover, we note that ${\tt TR}^\Lambda(\varphi,X)$ only imposes restrictions on numbers which code pairs $\la \xi,n\ra$ with $\xi \in |\Lambda|$ and says nothing about numbers not of this form.

In order to prove properties of transfinite recursion it will be useful to have restricted versions of ${\tt TR}^\Lambda(\varphi,X)$ whence we define
\[{\tt TR}_{\lambda}^\Lambda(\varphi,X) :=\forall \,  \xi \leq_\Lambda \lambda\, \forall x \ \big(x\in X_\xi\leftrightarrow \varphi(x,X_{<_\Lambda\xi})\big)\]
and ${\tt TR}_{<\lambda}^\Lambda(\varphi,X) :=\forall \,  \xi <_\Lambda \lambda\, {\tt TR}_{\lambda}^\Lambda(\varphi,X).$ The formula $\varphi$ may have many free number and set variables. However, one set variable of $\varphi$ plays a special role in ${\tt TR}^\Lambda(\varphi,X)$, like the variable $X$ in ${\tt TR}_{\lambda}^\Lambda(\varphi,X)$ above. We call this the \emph{recursion variable}. %If $X$ is this recursion variable, by ${\tt TR}^\Lambda(\varphi,Y)$ we shall denote ${\tt TR}^\Lambda(\varphi,X)[X/Y]$, that is, the formula where any occurrence of the recursion variable $X$ is replaced by an occurrence of $Y$: $\forall \,  \xi {<_\Lambda} \lambda\, \forall x \ \big(x\in Y_\xi\leftrightarrow \varphi(x,Y_{<_\Lambda\xi})\big)$.\david{No entiendo este comentario. \textquestiondown Qu\'e m\'as se podr\'ia entender por ${\tt TR}^\Lambda(\varphi,Y)$?}

Sometimes we may wish to emphasize that ${\tt TR}_{\lambda}^\Lambda(\varphi,Y)$ has other free variables appearing in $\varphi$. We will do so by using a semicolon as in ${\tt TR}_{\lambda}^\Lambda(\varphi,Y;X)$. We stipulate that any first-order variable, say $z$, in ${\tt TR}_{\lambda}^\Lambda(\varphi,Y;z)$ will be dotted when occurring under a box. Thus if, for example $z$ is free in $\varphi$, then ${\tt TR}_{\lambda}^\Lambda(\varphi,X)$ means ${\tt TR}_{\lambda}^\Lambda(\varphi,X;z)$ and by definition $[\xi|\Lambda]_T^\Lambda {\tt TR}_{\dot \lambda}^\Lambda(\varphi,Y)$ will denote $[\xi|\Lambda]_T^\Lambda {\tt TR}_{\dot \lambda}^{\overline \Lambda}(\varphi,Y; \dot z)$.

By convention we shall write the number variable which is universally quantified in ${\tt TR}_{\lambda}^\Lambda(\varphi,X)$ ($x$ in the above formulation) as the first in the list of variables. By $\equiv_{<_\Lambda \xi}$ we will denote set equality up to $\xi$, that is 
\[
X\equiv_{<_\Lambda \xi} Y \ :=\ \forall \, \zeta{<_\Lambda}\xi\, \forall x\ \Big(  \la \zeta,x \ra \in X \ \leftrightarrow \ \la \zeta,x \ra \in Y \Big).
\]
If the context allows us to, we shall simply write $X\equiv_{\xi} Y$ instead of $X\equiv_{<_\Lambda \xi} Y$. Let us dwell for a moment on how we can see the veracity of $\atr$ within second-order arithmetic. Given an arithmetical formula $\varphi$, could we directly \emph{define} a set $X$ so that ${\tt TR}^\Lambda(\varphi,X)$? 
%For a formula $\varphi$ we denote by $A^\varphi_\xi$ a set so that ${\tt TR}^\Lambda_\xi(\varphi,A^\varphi_\xi)$. 

Second-order arithmetic does not directly allow transfinite recursive definitions as in ${\tt TR}^\Lambda(\varphi,X)$, so that we cannot define $X_\xi$ out of $X_{<\xi}$ outright. However, we know that given a formula $\varphi$ the set $A$ such that ${\tt TR}^\Lambda(\varphi,A)$ holds is uniquely determined up to $\Lambda$. In particular, any initial part $A_{<\xi}$ is unique so that instead of recursively calling upon $A_{<\xi}$ we may use any set satisfying the defining properties so that the recursion is mimicked by a universal set quantifier. This reasoning can be formalized in $\aca$ for arithmetical $\varphi$:

\begin{lemma}\label{theorem:unicityTransfiniteRecursion}
Let $\varphi \in {\bm \Pi}^0_\omega$ be arithmetical. We have that
\[
\aca \vdash {\tt wo}(\Lambda) \ \to \ \forall \lambda \leq_\Lambda \Lambda\ \Big(\, {\tt TR}^\Lambda_{<\lambda}(\varphi,X) \wedge {\tt TR}^\Lambda_{<\lambda}(\varphi,Y) \ \to \ X\equiv_\lambda Y \Big).
\]
\end{lemma}

\begin{proof}
Reason in $\aca$ and assume ${\tt wo}(\Lambda)$. The claim follows by transfinite induction on the arithmetical formula $\, {\tt TR}^\Lambda_{<\lambda}(\varphi,X) \wedge {\tt TR}^\Lambda_{<\lambda}(\varphi,Y) \ \to \ X\equiv_\lambda Y$.
\end{proof}

The unicity proved by this lemma justifies the following definition.
\begin{definition}
\[
\widehat{\tt TR}^\Lambda_\lambda(\varphi,X) := \forall x \ \Big ( x_0 \leq_\Lambda \lambda \to \Big[ x\in X \ \leftrightarrow \ \forall Y\ \big( {\tt TR}^\Lambda_{x_0} (\varphi, Y) \to \varphi(x_1, Y_{<x_0}) \big) \Big]\Big)
\]\david{I had to change to $\leq_\Lambda$ because of the case where $\Lambda$ is a successor.}
\end{definition}

In particular, $\aca$ can prove that $\widehat{\tt TR}$ and ${\tt TR}$ are equivalent on arithmetical formulas.

\begin{lemma}\label{theorem:Pi11ATRversusATR}
Let $\varphi \in {\bm \Pi}^0_\omega$ be arithmetical. We have that
\[
\aca \vdash {\tt wo}(\Lambda) \ \to \ \forall \lambda < \Lambda\ \Big(  \widehat{\tt TR}^\Lambda_\lambda(\varphi,X) \ \leftrightarrow \ {\tt TR}^\Lambda_\lambda(\varphi,X) \Big).
\]
\end{lemma}

\begin{proof}
Reason in $\aca$ and assume ${\tt wo}(\Lambda)$. The implication ${\tt TR}^\Lambda_\lambda(\varphi,X) \to \widehat{\tt TR}^\Lambda_\lambda(\varphi,X)$ follows directly from Lemma \ref{theorem:unicityTransfiniteRecursion}. For the other direction, assume $\widehat{\tt TR}^\Lambda_\lambda(\varphi,X)$ and prove by induction on $\xi\leq \lambda$ that ${\tt TR}^\Lambda_\xi(\varphi,X)$.
\end{proof}

As a corollary we see that $\atr$ follows from ${\bm \Pi}^1_1$ comprehension.

\begin{corollary}
$\eca + {\bm \Pi}^1_1\mbox{-}\compax \vdash \atr$
\end{corollary}

\begin{proof}
Given some $\Lambda$ with ${\tt wo}(\Lambda)$, by ${\bm \Pi}^1_1\mbox{-}\compax$ we can form the set 
\[
\Big \{ x \mid x_0 \in |\Lambda| \wedge  \forall Y\ \big( {\tt TR}^\Lambda_{x_0} (\varphi, Y) \to \varphi(x_1, Y_{<x_0}) \big) \Big \}.
\qedhere\]
\end{proof}
The proof of this corollary reveals our proof strategy to see that $\atr$ is provable from the predicative consistency of \eca. If we have sufficient completeness and soundness, we can replace $\forall Y\, \big( {\tt TR}^\Lambda_{x_0} (\varphi, Y;X) \to \varphi(x_1, Y_{<x_0},X)\big)$ in the comprehension axiom by 
\begin{equation}\label{equation:PiOneOneButUnderABox}
\exists\, \xi < \Lambda\ [\xi |\Lambda, X]_T^\Lambda \, \big( \forall Y\ \big( {\tt TR}^\Lambda_{x_0} (\varphi, Y;\overline X) \to \varphi(x_1, Y_{<x_0}, \overline X) \big).
\end{equation} 
At first sight, this does not seem much of an improvement since $[\xi |\Lambda]_T^\Lambda \chi$ is $\Pi^1_1$ too: $\forall Z(\Provfor_{T|\Lambda}(Z)\to \provx \xi Z{\chi})$. However, under sufficient assumptions we can replace this by $\provx \xi P {\chi}$ for a particular iterated provability class $P$. Note that $\provx \xi P {\chi}$ is a ${\bm \Delta}^0_0$ formula! Moreover, we shall see that we can estimate the needed $\xi$ in \eqref{equation:PiOneOneButUnderABox} from $x$ so that we can even dispense of the existential quantifier in \eqref{equation:PiOneOneButUnderABox}.

A key ingredient that explains why this strategy will work over $\aca$ is that transfinite induction up to $\lambda<\Lambda$ can be mimicked in \eca by an $omega$-rule of nesting at most $\lambda$ as illustrated by the next lemma. Just as for transfinite recursion, let us consider restricted versions of transfinite induction:
\[
{\transin}^\Lambda_\lambda(\varphi) \ := \Big ( \forall \, \xi <_\Lambda \lambda \ \big(\forall \, \zeta <_\Lambda \xi \ \varphi(\zeta)\to \varphi(\xi)\big) \Big ) \to \forall \, \xi \leq_\Lambda  \lambda  \ \varphi (\xi).
\]
Although \eca cannot prove transfinite induction for arithmetical formulas, $\aca$ can prove that $\eca$ can prove it up to certain externally determined levels.

\begin{lemma}\label{theorem:acaProvesEcaProvesTI}
Let $\varphi \in \Pi^0_\omega(\Lambda, X)$ be an arithmetical formula. We have that
\[
\aca \vdash {\tt wo}(\Lambda) \to \forall \, \lambda <\Lambda\ [\lambda|\Lambda,X]^\Lambda_\eca {\transin}^{\overline{\Lambda}}_{\omega \cdot \dot \lambda} \big (\varphi(\overline X) \big ).
\]
\end{lemma}
\david{Removed the assumption $\lambda>0$, which is both ill-defined and unnecessary. Also, the use of $\omega\cdot\lambda$ may require some mention of multiplicative well-orders.}
\begin{proof}
Reason in $\aca$ and assume ${\tt wo}(\Lambda)$. Either there are no IPCs and the claim holds trivially, or there is a unique IPC, say $P$, and we prove $\forall \, \lambda < \Lambda\ [\lambda|\Lambda,X]_P {\transin}^{\overline{\Lambda}}_{\omega \cdot \dot \lambda}\big(\varphi(\overline{X})\big) $ by an easy induction on $\lambda$. 
\end{proof}

As a consequence, even though \eca cannot prove uniqueness as in Lemma \ref{theorem:unicityTransfiniteRecursion} nor the equivalence of $\widehat{\tt TR}^\Lambda_\lambda(\varphi,X)$ and ${\tt TR}^\Lambda_\lambda(\varphi,X)$ as in Lemma \ref{theorem:Pi11ATRversusATR}, under a box \eca can prove these facts:

\begin{corollary}
Let $\varphi \in \Pi^0_\omega(\Lambda, X)$ be an arithmetical formula. We have that
\begin{enumerate}
\item
$
\begin{aligned}[t]
\aca \vdash & {\tt wo}(\Lambda)  \to \\  
 & \forall \lambda < \Lambda\ [\lambda|\Lambda,X]^\Lambda_\eca \big(\, {\tt TR}^{\overline \Lambda}_{\dot \lambda}(\varphi,Y;\overline X) \wedge {\tt TR}^{\overline \Lambda}_{\dot \lambda}(\varphi,Z;\overline X)  \to  Y\equiv_{\dot \lambda} Z \big);
\end{aligned}
$
\item
$
\begin{aligned}[t]
\aca \vdash &{\tt wo}(\Lambda) \to\\
 &\forall \lambda {\leq_\Lambda} \Lambda\ [\lambda|\Lambda,X]^\Lambda_\eca  \Big(  \widehat{\tt TR}^{\overline \Lambda}_{\dot \lambda}(\varphi,Y;\overline X) \ \leftrightarrow \ {\tt TR}^{\overline \Lambda}_{\dot \lambda}(\varphi,Y;\overline X) \Big).
\end{aligned}
$
\end{enumerate}
\end{corollary}

\begin{proof}
Directly from Lemmas \ref{theorem:acaProvesEcaProvesTI}, \ref{theorem:Pi11ATRversusATR} and \ref{theorem:unicityTransfiniteRecursion}. Note that neither proof uses arithmetical comprehension but rather proceeds by transfinite induction. Clearly, one can bound the amount of transfinite induction by $\lambda$.
\end{proof}

{
We will now see that finite versions of transfinite recursion do not exceed the realm of $\aca$. For $\varphi(x,X) \in {\bm \Pi}^0_\omega$ let us recursively define $\varphi^{(0)}(x) := \bot$ and $\varphi^{(n+1)}(x,X) := \varphi(x,y\in X/\varphi^{(n)}(y))$. Thus, in the inductive step, we substitute any occurrence of $y\in X$ by $\varphi^{(n)}(y)$. Note that $x$ and $y$ are meant to be metavariables here and in particular we allow $x=y$. In a sense, these formulas $\varphi^{(n+1)}(x)$ tell all there is to know about finite recursion. 

\begin{lemma}\label{theorem:FiniteAtrIsInACA}
Let $\varphi(X)\in {\bm \Pi}^0_\omega$ be arithmetical. For each $n\in \omega$ we have that for any well-order $\Lambda$ with order type provably at least $n+1$ that\david{This is not enough. We also need $\eca$ to ``know'' the order type of finite elements.}
\[
\eca \vdash \bigvee_{m\leq n} 
% No overline over this n of course!
\big( x_0 {=} \overline m \, \wedge \,  \varphi^{(m+1)}(x_1) \big) \ \longleftrightarrow \ \forall X \big ( {\tt TR}^\Lambda_{\overline n+1} (\varphi, X) \wedge x\in X_{<\overline n{+}1}\big).
\]
\end{lemma}

\begin{proof}
By a simple external induction on $n$.
\end{proof}

As a consequence, we can prove a weak version of completeness for second-order information as is expressed in the next lemma.

\begin{lemma}\label{theorem:finiteRecursionCompleteness}
Let $\varphi \in {\Pi}^0_m (X,Y,\Lambda)$ with $m>0$ and $\psi(X,Y,\Lambda) \in{\Pi}^0_l (X,Y,\Lambda)$ be arithmetical. For any natural number $n>0$ we have that 
\begin{enumerate}
\item
$
\begin{aligned}[t]
\eca \vdash & {\tt wo}(\Lambda) \ \to \ \forall X \Big( \big ( {\tt TR}^\Lambda_{\overline n} (\varphi, X) \wedge x\in X_{<\overline n}\big) \ \to  \\
 & \ \ \ \ \ \ \ \ \ \ \  [\overline m\cdot \overline n|Y,\Lambda]^\Lambda_\eca \forall X\ \big ( {\tt TR}^{\overline \Lambda}_{\overline n} (\varphi, X;\overline Y) \to \dot x\in X_{<\overline n}\big)\Big)
\end{aligned}
$\\
and more generally,
\item
$
\begin{aligned}[t]
\eca \vdash & {\tt wo}(\Lambda) \ \to \ \Big( \big ( {\tt TR}^\Lambda_{\overline n} (\varphi, X) \wedge \psi( X_{<\overline n})\big) \ \to  \\
 & \ \ \ \  [\overline m\cdot \overline n + \overline{l}  |Y,\Lambda]^\Lambda_\eca \forall X\ \big ( {\tt TR}^{\overline \Lambda}_{\overline n} (\varphi, X;\overline Y) \to  \psi(X_{<\overline n}, \overline Y, \overline \Lambda)\big) \Big).
\end{aligned}
$
\end{enumerate}
\end{lemma}

\begin{proof}
It is easy to see that given $\varphi \in {\bm \Pi}^0_m$, the formula $\varphi^{(n+1)} \in {\bm \Sigma}^0_{m\cdot n+2}$. Thus, given ${\tt TR}^\Lambda_{\overline n} (\varphi, X)$ inside \eca, we can switch to the first-order equivalent as given by Lemma \ref{theorem:FiniteAtrIsInACA} with low enough complexity so that by provable completeness (Lemma \ref{LemmComplete}) we obtain this first-order equivalence under the box. Since under the box the sets $X$ that satisfy ${\tt TR}^\Lambda_{\overline n} (\varphi, X)$ are unique we obtain the required result. 
\end{proof} 

In the next subsection we shall see that the above lemma can simply be extended to transfinite values of $n$. Clearly, the route via a first-order equivalent of ${\tt TR}^\Lambda_{\overline n} (\varphi, X)$ will no longer work. However, it should not come as a surprise that the lemma can nonetheless be generalized since $\omega$-rules do provide us with a means to talk about infinite conjunctions and disjunctions in a sense.
\subsection{Predicative consistency proves predicative analysis}
\label{section:PredicativeConsistencyProvesPredicativeAnalysis}

In the previous subsection we sketched a strategy for how to prove $\atr$ from the predicative consistency of \eca. Thus, we would like to conclude that, given $x_0$ and $x_1$ there exists $\xi=\xi(x_0,x_1)$ such that the set 
\[
\{ x \mid x_0 \in |\Lambda| \wedge [\xi|\Lambda]_T^\Lambda \, \forall Y\ \big( {\tt TR}^\Lambda_{x_0} (\varphi, Y) \to \varphi(x_1, Y_{<x_0}) \big) \}.
\]
is equal --provably in $\eca + {\tt Pred\mbox{-}O\mbox{-}Cons}(\eca)$--  to the set 
\[
\{ x \mid x_0 \in |\Lambda| \wedge  \forall Y\ \big( {\tt TR}^\Lambda_{x_0} (\varphi, Y) \to \varphi(x_1, Y_{<x_0}) \big) \}.
\]
One inclusion turns out to be provable, in $\aca$, but not the other. However, we shall see that we can prove something very similar to equality.

First, we shall prove a lemma which gives us an estimate for $\xi(x_0,x_1)$. Note that by ${\bm \Sigma}^1_1$-completeness we have that $\eca + {\tt Pred\mbox{-}O\mbox{-}Cons}(\eca)$ proves
\begin{align*}
[\xi(x_0,x_1) |\Lambda]_T^\Lambda \, \forall Y\ \big( {\tt TR}^\Lambda_{x_0} (\varphi, Y) &\to \varphi(x_1, Y_{<x_0}) \big) \ \to\\
 &  \forall Y\ \big( {\tt TR}^\Lambda_{x_0} (\varphi, Y) \to \varphi(x_1, Y_{<x_0}) \big).
\end{align*}
The converse of this implication need not be provable in $\aca$ since the antecedent $\forall Y\ \big( {\tt TR}^\Lambda_{x_0} (\varphi, Y) \to \varphi(x_1, Y_{<x_0}) \big)$ is trivially satisfied when there is no such $Y$ at all. However, we can prove this implication if we know that at least one $Y$ exists satisfying ${\tt TR}^\Lambda_{x_0} (\varphi, Y)$. The next lemma tells us so, and gives us a concrete bound on $\xi(x_0,x_1)$. {The lemma is a generalization of the finite case as proven in Lemma \ref{theorem:finiteRecursionCompleteness}.}}

\begin{lemma}[$\aca$]\label{LemmCompleteTransfin}
Let $\Lambda$ be an additive ordinal and let $\varphi \in \Pi^0_m(X,Z,\Lambda)$ and $\psi\in{\Pi}^0_\ell (X,Z,\Lambda)$. Then,
\begin{align*}
{\tt wo}(\Lambda) \ \to  \forall &Z,X,\lambda<\Lambda\ \Big({\tt TR}^\Lambda_{<\lambda}(\varphi,Z)\wedge\psi(Z_{<\lambda})\\
& \rightarrow \provor {m{\cdot}\lambda+\ell}{X,{\Lambda}}T\, \forall Z\, \big({\tt TR}^{\overline \Lambda}_{\dot \lambda}(\varphi,Z; \overline X)\rightarrow \psi(Z_{<\dot \lambda}, \overline X, \overline \Lambda)\big)\Big).
\end{align*}
\end{lemma}

\proof
We reason in $\aca$, assume ${\tt wo}(\Lambda)$, fix $Z$ and $X$ and proceed to prove that the claim simultaneously for all subformulas of $\varphi$ and of $\psi$. We use (an arithmetical) induction on $\lambda$ with a subsidiary external induction on the build of $\psi$. The subsidiary induction on $\psi$ simply follows the proof of Lemma \ref{LemmComplete}. We only have to consider two new base cases.

First we consider the case when $\psi$ is of the form $x_0{<_\Lambda} \lambda\wedge  x \in Z$ in which case we have that $\overline {x_0}<_{\overline\Lambda}\overline\lambda$ is simply an axiom of $T|X,\Lambda$. Meanwhile, by the assumption that $x\in Z$ and ${\tt TR}^\Lambda_{\lambda}(\varphi,Z)$ we must have that both $\varphi(x_1,Z_{< x_0}, X)$ and ${\tt TR}^\Lambda_{x_0}(\varphi,Z)$ are  true and by the induction hypothesis applied to $x_0$ we have that
\[
\provor {m\cdot x_0+m}{X,\Lambda}T \, \forall Z\, \big({\tt TR}^{\overline \Lambda}_{\dot x_0}(\varphi,Z; \overline X)\rightarrow \varphi(\dot x_1, Z_{<\dot x_0}, \overline X, \overline \Lambda)\big).
%\forall Y\big({\tt TR}_{<_\Lambda}(\varphi(\overline X),Y)\rightarrow \varphi(Y_{{<_\Lambda} s})\big).
\]
Since $\Lambda$ is additive, we have that $m\cdot x_0+m \in |\Lambda|$.
Since $x_0 +1 \leq \lambda$ it now immediately follows that 
\[
\provor {m\cdot\lambda}{X,\Lambda}T \, \forall Z\, \big({\tt TR}^{\overline \Lambda}_{\dot \lambda}(\varphi,Z; \overline X)\rightarrow (\dot x_0<_{\overline \Lambda} \dot \lambda\wedge  \dot x \in Z) \big)
\]
as was to be shown. The case where $\psi=\neg(x_0{<_\Lambda} \lambda\wedge  x \in Z)$ is proven similarly.
\endproof

In the light of this lemma it makes sense to consider the following definition.

\begin{definition}
Let $\varphi \in \Pi^0_m(W, X,\Lambda)$ be arithmetical. We define a formula $\widetilde{\tt TR}^\Lambda_\lambda(\varphi,A)$ by
\begin{align*}
\forall x \ \bigg (& x_0 \leq_\Lambda  \lambda \to \Big( x \in A \\ &\leftrightarrow \, [m \cdot x_0 |W,\Lambda]^\Lambda_T\, \forall X\ \big( {\tt TR}^{\overline\Lambda}_{<\dot x_0} (\varphi, X;\overline W) \to \varphi(\dot x_1, X_{<\dot x_0}, \overline W, \overline \Lambda) \big) \Big)\bigg).
\end{align*}
We also define $\widetilde{\tt TR}^\Lambda_{<\lambda}(\varphi,A)$ by $\forall\eta<_\Lambda\lambda \ \widetilde{\tt TR}^\Lambda_{\eta}(\varphi,A)$ and $\widetilde{\tt TR}^\Lambda_{\Lambda}(\varphi,A)$ by $\forall\lambda<\Lambda \ \widetilde{\tt TR}^\Lambda_{\eta}(\varphi,A)$
\end{definition}

The notion of $\widetilde{\tt TR}^\Lambda_\lambda(\varphi,A)$ will be useful to us because, as we will see later, it is equivalent to ${\tt TR}^\Lambda_\lambda(\varphi,A)$; however, we only need reflection to construct a set satisfying $\widetilde{\tt TR}^\Lambda_\lambda(\varphi,A)$.

\begin{lemma}\label{LemmTildeExists}
Let $\varphi \in \Pi^0_\omega(X,Y,\Lambda)$ be arithmetical. We have that 
\[
\eca+{\tt Pred\mbox{-}O\mbox{-}Cons}(\eca) \vdash {\tt wo}(\Lambda) \to \forall X \exists Y\  \widetilde{\tt TR}^\Lambda(\varphi,Y;X).
\] 
\end{lemma}

\proof
Since $\eca+{\tt Pred\mbox{-}O\mbox{-}Cons}(\eca) \vdash \aca$ we know that any IPC will be unique given a well-order $\Lambda$. So, we reason in $\eca+{\tt Pred\mbox{-}O\mbox{-}Cons}(\eca)$, assume ${\tt wo}(\Lambda)$ and pick $X$ arbitrary as well as $\lambda \in |\Lambda|$. By $\provor 0{X,\mathord{\Lambda}}\eca\bot\to\bot$ we observe that $\exists P\  \Provfor_{\eca|X,\Lambda}(P)$.

So let $P$ be such an IPC and consider the set
\[
Y:=\{\langle \lambda,x\rangle \mid  \lambda \in |\Lambda| \, \wedge \, \provor{n\lambda} {X,\Lambda}P\ {\forall Y\, \big({\tt TR}_{\dot \lambda}^{\overline\Lambda}(\varphi,Y;X)\to \varphi(\dot x,Y_{<\dot\lambda}},\overline X, \overline \Lambda)\big)\}.
\] 
By ${\bm\Delta}^0_0$-comprehension $Y$ is a set and by definition and the uniqueness of an IPC we conclude that $\widetilde{\tt TR}_{\Lambda}(Z|X)$.\david{I think the ordinal operations will require $\aca$.}
%WE NEED TO CHECK THAT THE ORDINAL OPERATIONS DO NOT INCREASE THE FORMULA COMPLEXITY TO WHICH WE APPLY COMPREHENSION (not too important since we have access to $\aca$ but still...).
\endproof

\begin{lemma}\label{LemmTildeToTR}
For $\varphi \in \Pi^0_\omega(W, X,\Lambda)$ it is provable in $\eca +  {\tt Pred\mbox{-}O\mbox{-}Cons}(\eca)$ that if $A,W$ are sets and $\Lambda$ is a well-order then $\widetilde{\tt TR}^\Lambda(\varphi,A)$ implies that ${\tt TR}^\Lambda(\varphi,A)$.\david{There was a second lemma before with the opposite implication but we only need this one. Moreover, I removed the $\widehat{\tt TR}$ from the JSL version; we can leave it in the ArXiv as a didactic device, but it actually does not simplify the proofs, since you basically have to prove many things twice.}
\end{lemma}

\begin{proof} We will prove the more general claim that $\widetilde{\tt TR}^\Lambda_\lambda(\varphi,A)$ implies that ${\tt TR}^\Lambda_\lambda(\varphi,A)$ for all $\lambda<\Lambda$. Fix $\Lambda,A,W$ and $\varphi$ and proceed by transfinite induction on $\lambda<\Lambda$. Note that since all sets are fixed, this induction can be done in $\aca$.

So, assume that $\widetilde{\tt TR}^\Lambda_\lambda(\varphi,A)$ holds and, by induction, that for all $\eta<\lambda$, ${\tt TR}^\Lambda_\eta(\varphi,A)$. Suppose that $x=\langle \eta,n\rangle$ with $\eta\leq \lambda$. If $\eta< \lambda$, then we already have ${\tt TR}^\Lambda_{\eta} (\varphi, A)$ by the induction hypothesis, and thus $\forall x \ \big( x\in A\leftrightarrow \varphi(x,A_{<\eta})\big) $, so we may assume $\eta=\lambda$.

If $x\in A$, it is immediate by ${\bm \Pi}^1_1$-reflection that $\forall X\ \big( {\tt TR}^\Lambda_{<\lambda} (\varphi, X) \to \varphi(n, X_{<\lambda}) \big).$ But by the induction hypothesis, ${\tt TR}^\Lambda_{<\lambda} (\varphi, A)$ holds; therefore, in particular, $\varphi(n, A_{<\lambda})$ holds as well.

Conversely, if we have that $\varphi(n,A_{< \lambda})$, we wish to appeal to Lemma \ref{LemmCompleteTransfin} to conclude that $x\in A$. By the induction hypothesis, $A_{< \lambda}$ satisfies ${\tt TR}^{\Lambda}_{<\lambda}(\varphi,A_{< \lambda})$. Thus from Lemma \ref{LemmCompleteTransfin}, there is $m\in\mathbb N$ such that
\[[m \cdot x_0 |W,\Lambda]^\Lambda_T\, \forall X\ \big( {\tt TR}^{\overline\Lambda}_{<\dot x_0} (\varphi, X;\overline W) \to \varphi(\dot x_1, X_{<\dot x_0}, \overline W, \overline \Lambda) \big),\]
and hence $x\in A$. But we have now shown that, for all $x=\langle \eta,n\rangle$ with $\eta\leq \lambda$, $x\in A\leftrightarrow \varphi(n,A_{< \eta})$, that is, ${\tt TR}^\Lambda_\lambda(\varphi,A)$.
\end{proof}

We can now finally combine all our previous results and formulate the main theorem of this section.

\begin{theorem}\label{LemmRFNATR}
$\eca+{\tt Pred\mbox{-}O\mbox{-}Cons}({\eca}) \vdash {\rm ATR}_0$.
\end{theorem}

\proof
Since $\eca+{\tt Pred\mbox{-}O\mbox{-}Cons}({\eca}) \vdash \aca$ we have access to all the reasoning over $\aca$ and the result follows immediately from Lemmas \ref{LemmTildeExists} and \ref{LemmTildeToTR}.\david{I made some huge changes in this section since the proofs were incorrect and glossed over important steps. I removed many results from the JSL version which were interesting but not essential for our goal.}
%\ref{theorem:oracleConsistencyVersusOracleReflection} and Lemmas \ref{LemmTildeExists} and \ref{LemmTildeEquals}.
\endproof

\section{Countable coded $\omega$-models and reflection}\label{section:CountableCodedModels}

Our goal in this section is to derive a converse of Theorem \ref{LemmRFNATR}; in fact, we will even show that ${\rm ATR}_0$ extends ${\tt Pred\mbox{-}O\mbox{-}RFN}_{{\rm ACA}}[{\bm\Pi}^1_2]$. The main tool for this task will be the
notion of a \emph{countable coded $\omega$--model}. In what follows we shall
discuss existence results for $\omega$--models and the  satisfaction
definitions associated to them. First we briefly recall the
definition and basic properties of these models (we refer to
\cite[chapters VII and
VIII]{Simpson:2009:SubsystemsOfSecondOrderArithmetic} for a more detailed account of this topic).

\subsection{$\omega$--Models and satisfaction definitions}

Let us denote by $\flang$ the language of first-order arithmetic. A
structure for $\slang$, $\model=\langle \nmodel, {\mathcal S}\rangle$,
is given by an $\flang$--structure $\nmodel$ together with a
family ${\mathcal S}$ of subsets of the universe of $\nmodel$. An $\omega$--model is just an $\slang$--structure $\model=\langle \nmodel, {\mathcal S}\rangle$ where $\nmodel$ is the standard $\flang$--structure with universe $\omega$. Therefore, in order to fully describe an $\omega$--model it is enough to provide a subset $\mathcal S$ of ${\mathcal P}(\omega)$. This motivates the following definition within $\rca$.

\begin{definition}[$\rca$] A {\em countable coded $\omega$--model} is a set $\model\subseteq \mathbb N$ viewed as a code for a countable sequence of subsets of $\mathbb N$, $\{\modelpar n\mid n \in \mathbb N\}$, 
where for each $n\in \mathbb N$, $\modelpar n=\{i\mid \langle n,i\rangle\in \model \}.$
\end{definition}

A satisfaction  notion can be associated to each countable coded
$\omega$--model in a rather natural way. To this end we introduce some auxiliary concepts.

Working in $\rca$, for each  countable coded $\omega$--model $\model$, we denote by $\mlang$  the language obtained by adding to $\slang$ two sequences of new constant symbols:
\[\{c_n\mid n\in \mathbb N\}\mbox{ and  }\{C_n\mid n\in \mathbb{N}\}.\]
The constants $C_n$ are second-order and are used as names for the sets of the sequence coded by $\model$. On the other hand each constant $c_n$ will be interpreted as $n$, so by primitive recursion it can be shown within $\rca$ that there is a function $\val:\term\to \mathbb N$ that associates to each
closed first-order term $t$ of $\mlang$ its value $\val(t)$ under the usual
interpretation for the symbols of $\mlang$.\david{Le quit\'e a val la dependencia en $\model$.}

Now, always working within $\rca$, let $\snt$ denote the set of sentences of
$\mlang$. In what follows,  $\Gamma$ will denote  a set of $\mlang$--formulas closed under taking subformulas and substituting terms.

\begin{definition}[\rca]
Let $\model$ be a countable coded $\omega$--model. A
\emph{(full) satisfaction definition for $\model$} is a set $\sat\subseteq \mlang$ which obeys the usual recursive clauses of Tarski's truth definition, where each constant $c_n$ is interpreted as $n$ and constants $C_n$ are interpreted using
$\modelpar n$. In particular, for every $t\in \term$ and $n\in \mathbb N$,
\[
\begin{array}{rll}
(t\in C_n) \in \sat &\Leftrightarrow &\val(t)\in \modelpar n;\\
(\neg\varphi)\in\sat&\Leftrightarrow &\varphi\not\in\sat;\\
(\varphi_1\wedge \varphi_2) \in \sat & \Leftrightarrow & \varphi_1\in \sat
\mbox{ and }
\varphi_2 \in \sat;\\
(\forall u\,\varphi(u)) \in \sat & \Leftrightarrow &  \mbox{for all }n\in
\mathbb{N},\ \varphi(c_n)\in \sat;\\
(\forall X\,\varphi(X)) \in \sat & \Leftrightarrow &  \mbox{for all }n\in
\mathbb{N},\ \varphi(\modelpar n)\in \sat.
\end{array}
\]
We may assume other Booleans and quantifiers are defined in terms of $\neg,\wedge,\forall$.
We say that $\model$ is a \emph{full $\omega$--model} if
there exists a full satisfaction definition for $\model$.
\end{definition}

It can be  shown in $\rca$ that if $\model$ is a countable coded $\omega$--model, then there exists a (unique) partial satisfaction definition for $\Delta_0^0$-formulas in $\model$ (that is, more complex formulas are not necessarily assigned a truth value). Nevertheless, existence  of full satisfaction definition requires a stronger theory, such as $\atr$; uniqueness, on the other hand, does not require such a strong base theory. 

%(or the theories $\acapr$ or $\acapl$, see below). On the other hand, a partial satisfaction definitions are essentially unique and this can be shown in $\rca$.  

\begin{lemma}[$\rca$]
Let $\model$ be a  countable coded $\omega$--model. Then, there is at most one full satisfaction definition in $\model$.
\end{lemma}
% \begin{proof}
% First we define by bounded recursion a function
% $dg:Form(\mlang)\to \mathbb{N}$ such  that 
% $dg(\varphi)=\mbox{number of quantifiers and connectives ocurring in }\varphi$.
% Now, if $Sat_1$ and $Sat_2$ are partial satisfaction definitions for $\Gamma$ in
%   $\model$, then,  for each sentence $\varphi\in
%   \Gamma$, we can prove, by $\Sigma_0^0$--induction on $u$, that
% $$\forall u\,(dg(\varphi)\leq u \to (\varphi \in Sat_1 \leftrightarrow \varphi\in
% Sat_2))$$ 
% Therefore, $Sat_1=Sat_2$.
%\end{proof}

\begin{definition}[$\rca$]
  Let $\model$ be a countable coded $\omega$--model and let  $\varphi$ be a
  \emph{sentence}   of  $\mlang$. We say that $\model$ \emph{is a full $\omega$--model of $\varphi$} if there is a full satisfaction definition $Sat$ in $\model$ such that $\varphi\in Sat$, in which case we write $\model\models\varphi$. We say that $\model$ is a model of a set of  formulas $\Phi$  of $\mlang$ if, for every $\theta\in \Phi$, $\model$ is a model of the universal closure of $\theta$. 
\end{definition}

\begin{lemma}\label{model-reflection}
Let $\varphi(X_1,\dots,X_m,v_1,\dots,v_m)\in \Pi_\infty^0(X_1, \ldots, X_m)$ with all variables shown. We have that \rca proves:

For every full\david{Joost added specific notation for the satisfaction class, but I think this is not needed in view of the definition of {\em full} $\omega$-model.} countable coded $\omega$--model $\model$ and natural numbers $a_1,\hdots, a_n,$ $b_1,\hdots,  b_m$, we have that $
\varphi(\modelpar{a_1},\dots,\modelpar{a_n},b_1,\dots,b_m)$ holds if and only if $\model$ satisfies $\varphi(C_{a_1},\dots,C_{a_n},c_{b_1},\dots,c_{b_m})$.
\end{lemma}

\begin{proof}
  Straightforward by (external) induction on the syntactical complexity of the
  formula $\varphi$, using the Tarskian truth conditions.
\end{proof}

\begin{lemma}[$\rca$]\label{sat-full-induction} Let $\model$ be a full countable
  coded $\omega$--model. Then,
  \[\model\models \Robinson+ {\tt I}\Sigma^1_\omega.\]
\end{lemma}
\begin{proof}
Reasoning in $\rca$, let $Sat$ denote a   partial satisfaction definition for $\Gamma$ in $\model$.

Since $\Robinson$ is axiomatized by true $\Pi_1^0$--formulas, it follows from Lemma \ref{model-reflection} that $\model\models
\Robinson$. Now let $\varphi(u,X)\in \Gamma$ such  that for some $b$,
\[\model\models \varphi(0,\modelpar b) \land \forall u\,\big( \varphi(u,\modelpar b) \to
\varphi(u+1,\modelpar b)\big).\] 
Then
\[\theta(0,b,\varphi,Sat)\land \forall x\,\big(\theta(x,b,\varphi,Sat)\to
\theta(x+1,b,\varphi,Sat)\big),\]
where 
$\theta(x,b,\varphi,Sat)$ is the $\Delta^0_1$--formula $ \varphi(c_x,C_b)\in Sat.$
Since $\rca$ contains $\Sigma^0_1$ induction we see that $\forall
x\,\theta(x,b,\varphi,Sat)$ and, as a consequence,  $\model\models\forall
u\,\varphi(u,\modelpar b)$.
\end{proof}

\subsection{$\omega$--Models of $\aca$}

The following result will be very useful to us; see \cite[Theorem VIII.1.13]{Simpson:2009:SubsystemsOfSecondOrderArithmetic}.

\begin{proposition}[$\atr$]\label{PropACA0Mod} For each $X\subseteq \mathbb N$ there exists a unique, smallest, full countable coded $\omega$-model $\model$ such that $X\in\model$ and $\model\models \aca$. We will denote this model by $\model[X]$.
\end{proposition}

In view of Lemma \ref{sat-full-induction}, we immediately obtain the following:

\begin{corollary}[$\atr$]\label{CorACAMX}
Proposition \ref{PropACA0Mod} remains true if we replace $\aca$ by $\rm ACA$ (with full induction). In fact, we already have that $\model[X]\models{\rm ACA}$.
\end{corollary}

Previous results on countable coded $\omega$--models can be extended to
theories with an oracle as described at the end of Section 1. We only must
fix the  interpretation of the oracle. For each countable coded
$\omega$--model we will adopt the following convenion: the oracle
$\mathfrak{O}$ will always be interpreted using $\modelpar 0$. Similarly, we can
assume that in $\model[X]$, we have $X=\modelpar 0$.

\begin{lemma}[$\omega$--model soundness]\label{LemmOmSoundACA} The following is provable in $\atr$. Suppose that $T$ is representable, $X\subseteq \mathbb N$, $\Lambda$ is a well-order, $\model$ is a full $\omega$-model for $T$ with $\modelpar 0\equiv X$ and $\lambda\in|\Lambda|$ is such that $[\lambda|X]_{T}^\Lambda\varphi$. Then, $\model\models \varphi.$
\end{lemma}

\begin{proof}
Let us  fix a full satisfaction definition in $\model$, $Sat$, and let $P$ be an iterated provability class for $T$. Then it is enough to show that\[\forall X\,\forall \lambda< \Lambda\,\forall \varphi\,\big(\provx\lambda P\varphi\to \varphi\in Sat\big).\]
But this can be easily derived by fixing $X$ and proving by transfinite induction that we have
\[\forall \lambda < \Lambda \,\forall \varphi\,\big (\provx\lambda P\varphi\to \varphi\in Sat \big). \qedhere\]
\end{proof}

\subsection{Proving predicative reflection}

We are almost ready to state and prove our main theorem. We only need to prove the following lemma.

\begin{lemma}\label{LemmaATRRFN} Let $T\subseteq \slang$ be any formal theory such that it is provable in $\atr$ that every set $X$ can be included in a full $\omega$-model for $T$. Then, $\atr\vdash {\tt Pred\mbox{-}O\mbox{-}RFN}_{T}[{\bm \Pi}^1_2]$. \label{LemmATRRFN3}
 \end{lemma}

\proof 
Work in $\atr$. Let $\Lambda$ be such that ${\tt wo}(\Lambda)$; we must show that
\[\forall \lambda\in |\Lambda|,\ \lambda\text{-}{\tt OracleRFN}^\Lambda_{\rm ACA}[\bm\Pi_2^1].\]
Let $\varphi(X,x)\in {\bm \Pi_2^1}$ be such that for some set $A\subseteq \mathbb{N}$, $a\in \mathbb{N}$ and $\lambda\in |\Lambda|$, $[\lambda| A]_{T}^\Lambda \varphi(\bar{A}, \dot{a}).$ Since $\varphi(X,x)\in {{\bm \Pi}_2^1}$, we
can assume that $\varphi(X,x)$ is of the form $\forall Y\,\exists
Z\,\theta(X,Y,Z,x)$ for some formula  $\theta (X,Y,Z,x)\in \Pi_\omega^0$.

Let $B\subseteq \mathbb{N}$. We shall show that $\theta(A,B,C,a)$ holds for some set $C$. It follows from $[\lambda| A]_T^\Lambda \varphi(\bar{A}, \dot{a})$, that  
$[\lambda| A,B]_T^\Lambda \exists Z\,\theta(\bar{A},\bar{B},Z, \dot{a})$.
Hence, since we are working in $\atr$, there exists the least countable coded $\omega$--model $\model=\model[A,B]$ of $T$ containing $A$ and $B$, so that $A=\modelpar 0$ and
 $B=\modelpar 1$. By $\omega$--model soundness, $\model\models \exists Z\,\theta(\bar{A},\bar{B},Z,
c_a)$, and, since $B$ was arbitrary, we conclude that $\varphi$ is true.
\endproof

We may now summarize our results in our main theorem.

\begin{theorem}\label{TheoMain}
Let $U,T$ be computably enumerable theories such that $\eca\subseteq U\subseteq \atr$, $\eca\subseteq T$ and such that $\atr$ proves that any set $X$ can be included in a full $\omega$-model for $T$. Then,
\begin{equation}\label{EqMain}
\atr\equiv U+{\tt Pred\mbox{-}O\mbox{-}Cons}(T)\equiv U + {\tt Pred\mbox{-}O\mbox{-}RFN}_T[{\bm \Pi}^1_2].
\end{equation}
\end{theorem}

\proof
It is obvious that the third theory is at least as strong as the second, and the other inclusions are Theorem \ref{LemmRFNATR} and Lemma \ref{LemmaATRRFN}. 
\endproof

The following is then immediate in view of Corollary \ref{CorACAMX}:

\begin{corollary}\label{CorMain}
Let $\mathcal G=\{\eca,\rcaa,\rca,\aca\}$. Then, \eqref{EqMain} holds whenever $U\in \mathcal G\cup\{\atr\}$ and $T\in \mathcal G\cup \{{\rm ACA}\}.$
\end{corollary}

Although the following is well-known, it is an interesting consequence of our main result:

\begin{corollary}
$\atr$ is finitely axiomatizable.
\end{corollary}

\proof Immediate from Theorems \ref{TheoFinAx} and Corollary \ref{CorMain} setting $U=T=\aca$.

%	WE WILL FILL THIS OUT LATER
{\section*{Acknowledgements}
We would like to thank Jeremy Avigad, Enrique Casanovas, Carl Mummert, Michael Rathjen, Henry Towsner and Albert Visser for fruitful discussions and/or comments.}

{

}

\end{document}